\documentclass[11pt]{article}
\usepackage{amsmath,amssymb,amsthm,eucal, mathrsfs, mathtools,hyperref}
\usepackage{xcolor}
\usepackage[all]{xy}

\title{Trace formula and Levinson's theorem as an index pairing in the presence of resonances}
\author{Angus Alexander\thanks{email: 
\texttt{\href{mailto:angusa@uow.edu.au}{angusa@uow.edu.au}}}
\\[3pt]
School of Mathematics and 
Applied Statistics, University of Wollongong,\\
Wollongong, Australia\\
}

\advance\topmargin by -\headheight
\advance\topmargin by -\headsep
\evensidemargin=\oddsidemargin
\makeatletter
\def\section{\@startsection{section}{1}{\z@}{-3.5ex plus -1ex minus
  -.2ex}{2.3ex plus .2ex}{\large\bf}}
\def\subsection{\@startsection{subsection}{2}{\z@}{-3.25ex plus -1ex
  minus -.2ex}{1.5ex plus .2ex}{\normalsize\bf}}
\makeatother

\numberwithin{equation}{section} 

\theoremstyle{plain} 
\newtheorem{thm}{Theorem}[section]
\newtheorem{lemma}[thm]{Lemma}
\newtheorem{prop}[thm]{Proposition}
\newtheorem{cor}[thm]{Corollary}

\theoremstyle{definition} 
\newtheorem{defn}[thm]{Definition}
\newtheorem{ass}[thm]{Assumption}

\theoremstyle{remark} 
\newtheorem{rmk}[thm]{Remark}

\DeclareMathOperator{\Dom}{Dom}   

\newcommand{\eps}{\varepsilon} 

\newcommand{\B}{\mathcal{B}}  
\newcommand{\C}{\mathbb{C}}   
\renewcommand{\d}{\mathrm{d}} 

\newcommand{\e}{\mathrm{e}}
\newcommand{\F}{\mathcal{F}}  
\newcommand{\Fi}{\mathcal{F}} 
\renewcommand{\H}{\mathcal{H}}  
\newcommand{\N}{\mathbb{N}}   
\newcommand{\norm}[1]{\left|\left|#1\right|\right|} 
\newcommand{\R}{\mathbb{R}}   
\newcommand{\Sf}{\mathbb{S}}  

\newcommand{\vp}{\varphi}
\newcommand{\Z}{\mathbb{Z}} 

\newcommand{\stroke}{\mathbin|}     
\newcommand{\od}[2]{\frac{\mathrm{d}#1}{\mathrm{d}#2}}

\def\pairL_#1(#2|#3){{}_{#1}(#2\stroke#3)} 
\def\pairR(#1|#2)_#3{(#1\stroke#2)_{#3}} 
\def\scal<#1|#2>{\langle#1\stroke#2\rangle} 

\renewcommand{\epsilon}{\varepsilon}

\theoremstyle{definition}

\definecolor{MyBlue}{cmyk}{1,0.13,0,0.63}
\definecolor{MyGreen}{cmyk}{0.91,0,0.88,0.52}
\newcommand{\mylinkcolor}{MyBlue}
\newcommand{\mycitecolor}{MyGreen}
\newcommand{\myurlcolor}{black}

\usepackage{hyperref}
\hypersetup{%
  bookmarksnumbered=true,bookmarksopen=false,%
  plainpages=false,
  linktocpage=true,%
  colorlinks=true,breaklinks=true,%
  linkcolor=\mylinkcolor,citecolor=\mycitecolor,urlcolor=\myurlcolor,%
  pdfpagelayout=OneColumn,%
  pageanchor=true,%
}

\begin{document}

\maketitle

\vspace{-2pc}

\begin{abstract}
We realise the number of bound states of a Schr\"{o}dinger operator on $\R^n$ as an index pairing in all dimensions. Expanding on ideas of Guillop\'{e} and others, we use high-energy corrections to find representatives of the $K$-theory class of the scattering operator. These representatives allow us to compute the number of bound states using an integral formula involving heat kernel coefficients.
\end{abstract}
\maketitle

\parindent=0.0in
\parskip=0.00in


\parskip=0.06in

\section{Introduction}
\label{sec:intro}

In this paper we use the known form of the wave operator to construct an index pairing and prove Levinson's theorem for Schr\"{o}dinger operators on $\R^n$ in each dimension via the index of the wave operator. 

In dimension $n \leq 4$ there are correction terms required due to resonances and the low-energy behaviour of the scattering operator, whilst in dimensions $n \geq 2$ there are corrections required to account for the high-energy behaviour of the scattering operator. These correction terms are constructed explicitly using a pseudodifferential expansion of the resolvent and used to compute the index of the wave operator as a winding number, from which Levinson's theorem follows. 

As a consequence, we obtain the value of the spectral shift function at zero in all dimensions in the presence of any resonant behaviour. To the best of our knowledge, no such statement regarding the behaviour of the spectral shift function at zero in dimension $n = 4$ has appeared in the literature.

Levinson's original theorem \cite{levinson49} states that in dimension $n = 1$ the number of eigenvalues (counted with multiplicity) of the Schr\"{o}dinger operator $H = H_0+V$ (with $H_0 = -\Delta$) for a potential $V$ decaying sufficiently fast at infinity satisfies
\begin{align*}
N &= \frac{1}{\pi} (\delta(0)-\delta(\infty)) + \frac12 \nu,
\end{align*}
where $\delta$ is the scattering phase and $\nu \in \{0,1\}$ depends on the existence of a resonance (a distributional solution to $H\psi = 0$ with $\psi \notin L^2(\R)$). In higher dimensions, such a relation holds in each angular momentum mode as shown by Newton \cite{newton60} in dimension $n = 3$ (see also \cite{dong02} for higher dimensions). In addition, the work of Jensen \cite{jensen79, jensen80, jensen84}, Klaus and Simon \cite{KS80} and Boll\'{e} et al. \cite{bolle86, bolle88} shows that resonant behaviour can occur in dimension $n \leq 4$ only. The resonant contributions in dimension $n = 2,4$ to Levinson type results differ from those in dimension $n =1,3$ in that resonances contribute integer values rather than half-integers.  The differing behaviour in dimension $n = 2,4$ is due to the logarithmic singularity near $z = 0$ of the resolvent $(H-z)^{-1}$ as first observed by Klaus and Simon \cite{KS80} and Jensen \cite{jensen84} (see also \cite{bolle86, bolle88, jensen01} for additional detail). 

In Section \ref{sec:traces} we follow techniques of Colin de Verdi\`{e}re \cite{colin81} and Guillop\'{e} \cite{guillope81, guillope85} to construct a function $P_n: [0,\infty) \to \C$ given by
\begin{align*}
P_n(\lambda) &= 2\pi i\beta_n(V)+\sum_{j=1}^{\lfloor \frac{n-1}{2} \rfloor} C_j(n,V) \lambda^{\frac{n}{2}-j},
\end{align*}
where the $C_j(n,V), \beta_n(V)$ are heat-kernel type coefficients. In addition we show that the spectral shift function $\xi = \xi(\cdot,H,H_0)$ associated to the pair $(H,H_0)$ satisfies
\begin{align}\label{eq:asymptotics}
\lim_{\lambda \to \infty} \left( -2 \pi i \xi(\lambda) - P_n(\lambda) \right) &= 0.
\end{align}
Equation \eqref{eq:asymptotics} is proved by computing a resolvent expansion and using the limiting absorption principle to analyse the high-energy behaviour of the spectral shift function. As a result of Equation \eqref{eq:asymptotics} we can construct a different representative of the $K$-theory class $[S]$ of the scattering operator which converges to the identity at infinity in the trace norm.

Taking into account the relation $\textup{Tr}(S(\lambda)^*S'(\lambda)) = -2\pi i \xi'(\lambda)$, the general Levinson's theorem for the number of eigenvalues $N$ of $H = H_0+V$ (counted with multiplicity) is obtained by computing the index of the wave operator $W_-$ as a winding number and reads
\begin{align}\label{eq:levinson-intro}
- N &= \frac{1}{2\pi i} \int_0^\infty \left( \textup{Tr}(S(\lambda)^*S'(\lambda)) - P_n'(\lambda) \right) \, \d \lambda - \frac{1}{2\pi i} P_n(0) + N_{res},
\end{align}
where $N_{res}$ is a contribution related to the existence of resonances in dimension $n \leq 4$. The term $N_{res}$ is either $0$ or $\frac12$ in dimension $n = 1,3$, either $0$, $1$ or $2$ in dimension $n = 2$ and either $0$ or $1$ in dimension $n = 4$. Levinson's theorem in the form of Equation \eqref{eq:levinson-intro} was first proved in dimension $n = 3$ by Boll\'{e} and Osborn \cite{bolle77} and Newton \cite{newton77} and in dimension $n = 2$ by Boll\'{e} et al. in \cite{bolle86, bolle88}. In higher odd dimensions, Equation \eqref{eq:levinson-intro} is due to Guillop\'{e} \cite{guillope81, guillope85} and for all other dimensions to Jia, Nicoleau and Wang \cite{jia12}.

The computation of the index as a winding number relies on the wave operator $W_-$ having a particular form. In \cite[Theorem 5]{kellendonk08} in dimension $n = 1$, \cite[Theorem 1.3]{richard21} in dimension $n = 2$, \cite[Theorem 1.1]{kellendonk12} in dimension $n = 3$ and \cite[Theorem 3.1]{AR23} in dimension $n \geq 4$ it is shown that generically the wave operator $W_-$ is given by
\begin{align}\label{eq:wave-op-form-intro}
W_- &= \textup{Id}+\varphi(D_n)(S-\textup{Id}) +K,
\end{align}
where $K$ is a compact operator, $S$ denotes the scattering matrix, $D_n$ denotes the generator of the dilation group on $L^2(\R^n)$ and
\begin{align*}
\varphi(D_n) &= \frac12 \left( \textup{Id}+\tanh{(\pi D_n)} - i A \cosh{(\pi D_n)}^{-1} \right),
\end{align*}
where $A$ is a self-adjoint involution commuting with $D_1$ if $n = 1$ and $A = \textup{Id}$ otherwise. In dimension $n = 2$ in the presence of $p$-resonances and dimension $n = 4$ in the presence of resonances (see Definition \ref{defn:resonances}) it has been shown recently (\cite[Lemma 4.3]{ANRR} and \cite[Theorem 3.1]{AR23-4D}) that Equation \eqref{eq:wave-op-form-intro} needs to be modified by an additional Fredholm operator, see Theorem \ref{thm:wave-op-form}.

Much work has been done in recent years in developing similar formulae to Equation \eqref{eq:wave-op-form-intro} in various scattering contexts. These include Schr\"{o}dinger scattering \cite{ANRR, AR23, AR23-4D, kellendonk06, kellendonk08, richard13, richard13ii, richard21}, rank-one perturbations \cite{richard10}, the Friedrichs-Faddeev model \cite{isozaki12}, Aharanov-Bohm operators \cite{kellendonk11}, lattice scattering \cite{bellissard12}, half-line scattering \cite{inoue20}, discrete scattering \cite{inoue19}, scattering for inverse square potentials \cite{DR, inoue19ii} and for Dirac operators with zero-range interactions \cite{richard14}. 

In \cite[Theorem 4.7]{AR23} it was shown that if $S(0) = \textup{Id}$ and the wave operator $W_-$ is of the form of Equation \eqref{eq:wave-op-form-intro}, then there is a pairing between the $K$-theory class $[S] \in K_1(C_0(\R^+) \otimes \mathcal{K}(L^2(\Sf^{n-1})))$ and the $K$-homology class $[D_+] \in K^1(C_0(\R^+) \otimes \mathcal{K}(L^2(\Sf^{n-1})))$, where $D_+$ denotes the generator of the dilation group on $L^2(\R^+) \otimes L^2(\Sf^{n-1})$. In addition,  \cite[Theorem 4.7]{AR23} shows that this index pairing is computed via
\begin{align*}
\langle [S], [D_+] \rangle &= - \textup{Index}(W_-).
\end{align*}
The proof of \cite[Theorem 4.7]{AR23} generalises to a class of `admissable' unitary operators $U$ on $L^2(\R^n)$ commuting with $H_0$ which, when considered in the spectral representation of $H_0$, satisfy $\displaystyle U(0) =  \lim_{\lambda \to \infty} U(\lambda) = \textup{Id}$. In this case we show in Section \ref{sec:pairing} that we can construct the operator 
\begin{align*}
W_U &= \textup{Id}+\varphi(D_n)(U-\textup{Id})
\end{align*}
and as a consequence we have the pairing
\begin{align*}
\langle [U], [D_+] \rangle &= - \textup{Index}(W_U).
\end{align*}
In dimension $n = 1,3$ it is not always the case that $S(0) = \textup{Id}$ and thus we introduce an additional unitary $\sigma$ so that $S\sigma^*$ defines an admissable unitary. 

Even for an admissable unitary, the limit $\displaystyle \lim_{\lambda \to \infty} U(\lambda) = \textup{Id}$ holds only in operator norm and not necessarily in trace norm. The pairing gives us the flexibility to choose different representatives of the class $[S]$ with better trace norm behaviour at infinity. The remainder of the paper uses this flexibility to find a representative of $[S]$ that converges to the identity at infinity in trace norm.

The paper is organised as follows. In Section \ref{sec:prelim} we fix notation and recall a number of results from scattering theory. In particular the definition of resonances, the high and low energy behaviour of the scattering operator, the spectral shift function and the form of the wave operator. In Section \ref{sec:pairing} we define a class of Fredholm operators $(W_U) \subset \B(\H)$ parametrised by particular admissable unitaries $U \in \B(\H)$ and show that the pairing between the class $[U] \in K_1(C_c0(\R^+) \otimes \mathcal{K}(L^2(\Sf^{n-1})))$ and $[D_+]$ is computed by the Fredholm index of $W_U$. Generically, $S$ is an admissable unitary and the class $[S]$ can be paired with $[D_+]$ to compute the Fredholm index of $W_-$. When $S(0) \neq \textup{Id}$, an additional unitary is required to account for this low-energy behaviour.

In Section \ref{sec:traces} we use the limiting absorption principle to compute the high-energy behaviour of the spectral shift function in terms of particular traces of resolvents. The high-energy behaviour of the spectral shift function is given by a polynomial in the energy, whose coefficients we show are related to the coefficients of a heat-type expansion. Finally, in Section \ref{sec:corrections} we construct explicit low-energy corrections to the scattering operator in dimension $n = 1,3$ and high-energy corrections to the scattering operator in dimensions $n \geq 2$, and thus prove Levinson's theorem in all dimensions by computing the Fredholm index of $W_-$ as a winding number. As a consequence, we determine the behaviour of the spectral shift function at zero in each dimension, accounting for the presence of resonances.

{\bf Acknowledgements:} 
I would like to thank Adam Rennie for his consistent insight and guidance during the thesis from which this work originated, as well as for his careful reading of earlier versions of this work. I would also like to gratefully acknowledge numerous enlightening conversations with Alan Carey, Galina Levitina and Serge Richard on the topic of scattering theory. This work was completed with the support of an Australian Government Research Training Program (RTP) Scholarship and the ARC Discovery grant DP220101196.


\section{Preliminaries on scattering theory} \label{sec:prelim}

Throughout this document we will consider the scattering theory on $\R^n$ associated to the operators
\[
H_0=-\sum_{j=1}^n\frac{\partial^2}{\partial x_j^2}=-\Delta
\quad\mbox{and}\quad
H=-\sum_{j=1}^n\frac{\partial^2}{\partial x_j^2}+V
\]
where the (multiplication operator by the) 
potential $V$ is real-valued and satisfies 
\begin{align}
\label{eq:ass11}
|V(x)| &\leq C(1+|x|)^{-\rho}
\end{align}
for various values of $\rho$ depending on the particular result and the dimension $n$. 
We denote the Schwartz space $\mathcal{S}(\R^n)$ and its dual $\mathcal{S}'(\R^n)$ and recall the weighted Sobolev spaces 
\[
H^{s,t}(\R^n)= \Big\{f \in \mathcal{S}'(\R^n): \norm{f}_{H^{s,t}} := \norm{(1+|x|^2)^{\frac{t}{2}} (\textup{Id}-\Delta)^\frac{s}{2} f} < \infty \Big\}
\] 
with index $s \in \R$ indicating derivatives and $t \in \R$ associated to decay at infinity \cite[Section 4.1]{amrein96}. With $\langle\cdot,\cdot\rangle$ the Euclidean inner product on $\R^n$, we denote the Fourier transform by
\[
\F:L^2(\R^n)\to L^2(\R^n),\qquad [\F f](\xi)=(2\pi)^{-\frac{n}{2}}\int_{\R^n}e^{-i\langle x,\xi\rangle}f(x)\,\d x.
\]
Note that the Fourier transform $\F$ is an isomorphism from $H^{s,t}$ to $H^{t,s}$ for any $s,t \in \R$. We will frequently drop the reference to the space $\R^n$ for simplicity of notation. 

We denote by $\mathcal{B}(\H_1,\H_2)$ and $\mathcal{K}(\H_1,\H_2)$ the bounded and compact operators from $\H_1$ to $\H_2$ and for $p \in \N$ we write $\mathcal{L}^p(\H_1,\H_2)$ for the $p$-th Schatten classes. In particular we have the trace class operators $\mathcal{L}^1(\H_1,\H_2)$ and the Hilbert-Schmidt operators $\mathcal{L}^2(\H_1,\H_2)$. For $p \in \N$ and $T \in \mathcal{L}^p(\H_1)$ we write $\textup{Det}_p(\textup{Id}+T)$ for the $p$-th regularised determinant of $\textup{Id}+T$ (see \cite{simon79} for more details). For $p =  1$ we obtain the Fredholm determinant.
For $z \in \C \setminus \R$, we let
\[
R_0(z)=(H_0-z)^{-1},\qquad R(z)=(H-z)^{-1}
\]
and the boundary values of the resolvent are defined as
\begin{align}
R_0(\lambda \pm i0) &= \lim_{\eps \to 0}{R_0(\lambda \pm i \eps)}\quad\mbox{and}\quad R(\lambda \pm i0) = \lim_{\eps \to 0}{R(\lambda \pm i \eps)}.
\end{align}
The limiting absorption principle \cite[Theorems 4.1 and 4.2]{agmon75} tells us that these boundary values exist in $\mathcal{B}(H^{0,t},H^{2,-t})$ for any $t > \frac12$ and $\lambda \in (0,\infty)$. The operator $H_0$ has purely absolutely continuous spectrum, and in particular no kernel. The operator $H$ can have eigenvalues and for $V$ satisfying Assumption \eqref{eq:ass11} with $\rho > 1$ we have that these eigenvalues are negative, or zero \cite[Theorem 6.1.1]{yafaev10}. 

The one-parameter unitary group of dilations on $L^2(\R^n)$ is
given on $f\in L^2(\R^n)$ by
\begin{align}\label{defn:dilation}
[U_n(t)f](x) &= \e^{\frac{nt}{2}} f(\e^t x),\qquad t\in\R.
\end{align}
We denote the self-adjoint generator of $U_n$ by $D_n$.
The generator of the group $(U_+(t))$ of dilations on the half-line $\R^+$ is denoted $D_+$
(which is $D_1$ restricted to the positive half-line). The generators of the dilation groups are given by
\begin{align}\label{eq:dilation-defn}
D_+ &= \frac{y}{i} \od{}{y} + \frac{1}{2i}{\rm Id},\qquad D_n=\sum_{j=1}^n\frac{x_j}{i}\frac{\partial}{\partial x_j}+\frac{n}{2i}{\rm Id}.
\end{align}
Since each of $D_+,D_n$ generate one-parameter groups, we can recognise functions of these operators. For $D_n$ and $\vp:\R\to\C$ a bounded function whose Fourier transform has rapid decay and $g \in C_c^\infty(\R^n)$, we have
\begin{align*}
[\vp(D_n) g](\rho) &= (2\pi)^{-\frac12} \int_\R{[\Fi^* \vp](t) \e^{\frac{nt}{2}} g(\e^t \rho)\, \d t},
\end{align*}
with a similar formula for $D_+$.

Several Hilbert spaces recur, and we adopt the notation (following \cite[Section 2]{jensen81} for the relations between the spaces and operators we introduce here)
\[
\H = L^2(\R^n),\quad \mathcal{P} = L^2(\Sf^{n-1}),\quad \H_{spec} = L^2(\R^+, \mathcal{P}) \cong L^2(\R^+) \otimes \mathcal{P}.
\] 
Here $\H_{spec}$ provides the Hilbert space on which we can diagonalise the free Hamiltonian $H_0$.

Since $V$ is bounded, $H = H_0+V$ is self-adjoint with $\Dom(H) = \Dom(H_0)$. The wave operators 
\[
W_\pm=\mathop{\textup{s-lim}}_{t\to\pm\infty}e^{itH}e^{-itH_0}
\]
exist and are asymptotically complete if $\rho > 1$  \cite[Theorem 1.6.2]{yafaev10}. 
The wave operators are partial isometries satisfying $W_\pm^*W_\pm = \textup{Id}$ and $W_\pm W_\pm^* = P_{ac}$, the projection onto the absolutely continuous subspace for $H$. In addition, the wave operators are Fredholm operators and their index can be determined directly in terms of the eigenvalues of $H$.
\begin{prop}\label{prop:index-wave-op}
Suppose that $V$ satisfies \eqref{eq:ass11} for some $\rho > 1$. Let $N$ be the total number of eigenvalues of $H = H_0+V$, counted with multiplicity. Then
\begin{align*}
\textup{Index}(W_\pm) &= -N.
\end{align*}
\end{prop}
\begin{proof}
This follows from the relations $W_\pm^*W_\pm = \textup{Id}$ and $W_\pm W_\pm^* = \textup{Id}-P_p$, where $P_p$ is the projection onto the point spectrum of $H$.
\end{proof}
For $\rho > 1$, the scattering operator is the unitary operator
\begin{align}
S &= W_+^*W_-,
\end{align}
which commutes strongly with the free Hamiltonian $H_0$.
For our analysis of the scattering operator, we introduce the unitary which diagonalise the free Hamiltonian. 

\begin{defn}
\label{def:diag}
For  $\lambda \in \R^+$, $s \in \R$ and $t > \frac12$ we define the operator 
\[
\Gamma_0(\lambda): H^{s,t} \to \mathcal{P}\quad \mbox{by}\quad
[\Gamma_0(\lambda) f](\omega) = 2^{-\frac12} \lambda^{\frac{n-2}{4}} [\F f](\lambda^\frac12 \omega)
\]
and the operator which diagonalises the free Hamiltonian $H_0$ as
\[
F_0: \H \to \H_{spec}\quad \mbox{by} \quad [F_0 f](\lambda,\omega) = [\Gamma_0(\lambda) f](\omega).
\] 
\end{defn}
By \cite[Theorem 2.4.3]{kuroda78} the operator $\Gamma_0(\lambda)$ extends to a bounded operator on $H^{s,t}$ for all $s \in \R$ and $t > \frac12$.
We recall from \cite[Equation (3.5)]{AR23} that for $\lambda > 0$, $t > \frac{n}{2}$ and $f \in L^2(\R^n)$ we have the estimate
\begin{align}\label{eq:estimate-ref}
\norm{\Gamma_0(\lambda) f}_{\mathcal{P}} &\leq C \lambda^{-\frac{1}{4}} \norm{f}_{H^{0,t}}
\end{align}
for some constant $C > 0$.

\begin{lemma}[{\cite[p. 439]{jensen81}}]
\label{lem:eff-emm}
The operator $F_0$ is unitary. Moreover for $\lambda \in[0,\infty)$, $\omega\in \Sf^{n-1}$ and $f \in \H_{spec}$ we have
\[
[F_0H_0F_0^* f](\lambda,\omega)=\lambda f(\lambda,\omega)=:[Mf](\lambda,\omega).
\] 
\end{lemma}
Here we have defined the operator $M$ of multiplication by the spectral variable.
As a consequence of the relation $SH_0 = H_0 S$, there exists a family $\{S(\lambda) \}_{\lambda \in \R^+}$ of unitary operators on $\mathcal{P} = L^2(\Sf^{n-1})$ such that for all $\lambda \in \R^+, \omega \in \Sf^{n-1}$ and $f \in \H$ we have
\begin{align*}
[F_0 S f](\lambda,\omega) &= S(\lambda) [F_0 f](\lambda,\omega).
\end{align*}
For historical reasons, we refer to $S(\lambda)$ as the scattering matrix at energy $\lambda \in \R^+$

Later we will need the following decay assumptions on the potential. 
\begin{ass}
\label{ass:best-ass}
In dimension $n$ we fix $\rho$ such that if $n = 1$ then $\rho > \frac{5}{2}$;  if $n=2$ then $\rho > 11$; if $n=3$ then $\rho > 5$; if $n=4$ then $\rho > 12$; and if $n \geq 5$ then $\rho > \frac{3n+4}{2}$. We assume that $|V(x)| \leq C(1+|x|)^{-\rho}$ for almost all $x \in \R^n$.
\end{ass}

We now define zero-energy resonances, a low-energy phenomena known to provide obstructions to generic behaviour in scattering theory in low dimensions.

\begin{defn}\label{defn:resonances}
Suppose that $V$ satisfies Assumption \ref{ass:best-ass}. If $n \neq 2$ we say there is an $s$-resonance if there exists a non-zero bounded distributional solution to $H\psi = 0$. If $n = 2$ we say there is a $p$-resonance if there exists a non-zero distributional solution $\psi$ to $H\psi = 0$ with $\psi \in L^q(\R^2) \cap L^\infty(\R^2)$ for some $q > 2$. We say that there is an $s$-resonance if there exists a non-zero bounded distributional solution $\psi$ to $H\psi = 0$ with $\psi \notin L^q(\R^2)$ for all $q < \infty$.
\end{defn}
General bounds on the resolvent of $H$ \cite{jensen80} show that there can be no resonances for dimension $n \geq 5$ (see also the analysis in \cite{goldberg15, goldberg17}). The naming convention for the types of resonances is due to the fact that for spherically symmetric potentials, $s$-resonances can only occur with angular momentum zero and $p$-resonances with angular momentum one (see \cite{bolle86, bolle88, gibson86, KS80}).

\subsection{The scattering matrix and the spectral shift function}
The following result describes the low energy behaviour of the scattering matrix in the norm of $\B(\H)$. The low energy behaviour of the scattering operator has been obtained in dimension $n = 1$ in \cite[Theorem 4.1]{bolle85} (see also \cite[Proposition 9]{kellendonk08} and \cite[Theorem 2.15]{zworski19}), in dimension $n = 2$ in \cite[Theorem 4.3]{bolle88} and \cite[Theorem 1.1]{richard21}, in dimension $n = 3$ in \cite[Section 5]{jensen79} and in dimension $n \geq 4$ in \cite[Theorem 2.15]{AR23}. 
\begin{thm}\label{thm:scat-mat-zero}
Suppose that $V$ satisfies Assumption \ref{ass:best-ass}. Then for $n \neq 1,3$ we have $S(0) = \textup{Id}$. For $n = 3$ we have $S(0) = \textup{Id}-2P_s$, where $P_s$ is the projection onto spherical harmonics of order zero in $L^2(\Sf^2)$ if there exists a resonance and $P_s = 0$ otherwise. For $n = 1$ we have
\begin{align*}
S(0) &= \begin{pmatrix} 0 & -1 \\ -1 & 0 \end{pmatrix}
\end{align*}
if there does not exist a resonance and if there does exist a resonance, there exists $c_+, c_- \in \R \setminus\{0\}$ with $c_+^2+c_-^2 = 1$ and
\begin{align*}
S(0) &= \begin{pmatrix} 2 c_+ c_- & c_+^2-c_-^2 \\ c_-^2 - c_+^2 & 2 c_+ c_- \end{pmatrix}.
\end{align*}
\end{thm}
\begin{rmk}
The values $c_\pm$ present in dimension $n = 1$ correspond to the values $\displaystyle \lim_{x \to \pm \infty} \psi(x)$ of an appropriately normalised resonance $\psi$. See \cite[Theorem 2.15]{zworski19} and \cite[Section 3]{jensen01} for more details. The value of $S(0)$ in dimension $n = 1$ is also dependent on a choice of basis for $M_2(\C)$, with a different choice of basis used to compute $S(0)$ in \cite[Proposition 9]{kellendonk08}.
\end{rmk}
We summarise below some additional useful properties of the scattering matrix \cite[Proposition 1.8.1 and Proposition 8.1.9]{yafaev10}.
\begin{thm}\label{thm:scat-properties}
Suppose that $V$ satisfies Assumption \eqref{eq:ass11} for some $\rho > 1$. The scattering matrix $S(\lambda)$ is given for all $\lambda \in \R^+$ by the equation
\begin{align}\label{eq: scat matrix defn}
S(\lambda) &= \textup{Id} - 2\pi i \Gamma_0(\lambda)(V-VR(\lambda+i0)V)\Gamma_0(\lambda)^*.
\end{align}
For each $\lambda\in\R^+$, the operator $S(\lambda)$ is unitary in $\mathcal{P}=L^2(\Sf^{n-1})$ and depends norm continuously on $\lambda \in \R^+$. Furthermore, if $\rho > n$ then $S(\lambda)-\textup{Id} \in \mathcal{L}^1(\H)$ and is differentiable in the norm of $\mathcal{L}^1(\H)$.
\end{thm}
We now describe the high energy behaviour of the scattering matrix in $\B(\mathcal{P})$, which has been determined in dimension $n = 1$ in \cite{bolle85}, in dimension $n = 2$ in \cite[Theorem 1.3]{richard21}, in $n = 3$ in \cite[Lemmas 2.2 and 2.4]{richard13} and in \cite[Corollary 3.10]{AR23} for $n \geq 4$.
\begin{lemma}\label{lem:S-infty}
Suppose that $V$ satisfies Assumption \ref{ass:best-ass}. Then $\displaystyle \lim_{\lambda \to \infty} S(\lambda) = \textup{Id}$, with the limit taken in the norm of $\B(\mathcal{P})$.
\end{lemma}
The result of Lemma \ref{lem:S-infty} is not true when taken in more refined topologies. Since for all $\lambda \in \R^+$ we have $S(\lambda)-\textup{Id} \in \mathcal{L}^1(\mathcal{P})$ the determinant of $S(\lambda)$ is well-defined. We have the  following more precise statement regarding the properties of $\textup{Det}(S(\lambda))$ due to Guillop\'{e} \cite[Theorem III.1]{guillope81}.
\begin{lemma}\label{lem:guillope-det-S}
Suppose that $V = q_1 q_2$ with $q_1,q_2 \in C_c^\infty(\R^n)$. Then for all $\lambda > 0$ and $p \geq \lfloor \frac{n}{2} \rfloor$ we have
\begin{align*}
&\textup{Det} \left(S(\lambda) \right) \\
&= \frac{\textup{Det}_p \left( \textup{Id}+q_1 R_0(\lambda - i 0 ) q_2 \right)}{\textup{Det}_p \left( \textup{Id}+q_1 R_0(\lambda+i0) q_2 \right)} \exp \left( \sum_{\ell = 1}^{p-1} \frac{(-1)^\ell}{\ell} \textup{Tr} \left( \left(q_1 R_0(\lambda+i0) q_2 \right)^\ell - \left(q_1 R_0(\lambda-i0) q_2 \right)^\ell \right) \right).
\end{align*}
\end{lemma}
\begin{rmk}
If $n = 1$ and $p = 1$ or $n = 2,3$ and $p = 2$ we only require that $V$ satisfies Assumption \ref{ass:best-ass} to obtain the statement of Lemma \ref{lem:guillope-det-S} (see \cite[Section 9.1]{yafaev10}).
\end{rmk}
We also recall the following \cite[Proposition 9.1.3]{yafaev10}.
\begin{lemma}\label{lem:det-limits}
Suppose that $\rho$ satisfies Assumption \ref{ass:best-ass}. If $n = 1$, let $p \geq 1$, if $n = 2,3$ let $p \geq 2$ and if $n \geq 4$ let $p \geq n$. Define for $z \in \C \setminus \R$ the function
\begin{align*}
D_p(z) &= \textup{Det}_p(\textup{Id}+q_1 R_0(z) q_2).
\end{align*}
Then we have
\begin{align*}
\lim_{|z| \to \infty} D_p(z) &= 1
\end{align*}
uniformly in $\textup{Arg}(z)$. By the limiting absorption principle this limit extends to the positive real axis also.
\end{lemma}

We now recall the spectral shift function \cite{birman62, krein53} for the pair $(H,H_0)$ and some of its defining properties (see \cite[Section 9.1]{yafaev10}).

\begin{thm}\label{thm:ssf-properties}
Suppose that $V$ satisfies Assumption \ref{ass:best-ass} and let $S$ be the corresponding scattering operator. Then there exists a unique (up to an additive constant) real-valued function $\xi(\cdot,H,H_0): \R \to \R$ such that 
\begin{align}\label{eq:ssf-defining}
\textup{Tr}(f(H)-f(H_0)) &= - \int_\R \xi(\lambda, H, H_0) f'(\lambda) \, \d \lambda,
\end{align}
at least for all $f \in C_c^\infty(\R)$. We specify $\xi(\cdot, H, H_0)$ uniquely by the convention $\xi(\lambda, H, H_0) = 0$ for $\lambda$ sufficiently negative and thus $\xi(\cdot,H,H_0)$ satisfies for $\lambda < 0$ the relation
\begin{align*}
\xi(\lambda,H,H_0) &= -\sum_{k=1}^K M(\lambda_k) \chi_{[\lambda_k,\infty)}(\lambda),
\end{align*}
where we we have indexed the distinct eigenvalues of $H$ as $\lambda_1 < \cdots < \lambda_K \leq 0$ and each $\lambda_j$ has multiplicity $M(\lambda_j)$. Furthermore, we have $\xi(\cdot,H,H_0) \vert_{(0,\infty)} \in C^1(0,\infty)$ and for $\lambda > 0$ the relations
\begin{align*}
\textup{Det}(S(\lambda)) &= \e^{-2\pi i \xi(\lambda)} \quad \textup{ and } \quad \textup{Tr}\left(S(\lambda)^*S'(\lambda)\right) = -2\pi i \xi'(\lambda)
\end{align*}
hold.
\end{thm}

We call $\xi(\cdot,H,H_0)$ the spectral shift function for the pair $(H,H_0)$ and will often just write $\xi = \xi(\cdot,H,H_0)$. Using integration by parts we can rewrite the definining property \eqref{eq:ssf-defining}, obtaining the Birman-Kre\u{\i}n trace formula.

\begin{lemma}\label{lem:birman-krein-ibp}
Suppose that $V$ satisfies Assumption \ref{ass:best-ass} and let $S, \xi$ be the corresponding scattering operator and spectral shift function. Then for all $f \in C_c^\infty(\R)$ we have
\begin{align*}
\textup{Tr}(f(H)-f(H_0)) &= \frac{1}{2\pi i} \int_0^\infty f(\lambda) \textup{Tr}\left(S(\lambda)^*S'(\lambda) \right) \, \d \lambda  + \sum_{k=1}^K f(\lambda_k) M(\lambda_k) \\
&\quad + f(0) \left( \xi(0-)-\xi(0+)-M(0) \right),
\end{align*}
where we have defined $\displaystyle \xi(0\pm) = \lim_{\eps \to 0^+} \xi(\pm \varepsilon)$.
\end{lemma}

In fact by Theorem \ref{thm:ssf-properties} we have, with $N$ the total number of eigenvalues of $H$ counted with multiplicity, the relation $\xi(0-) = -N+M(0)$. We can then rewrite the Birman-Kre\u{\i}n trace formula as
\begin{align*}
\textup{Tr}(f(H)-f(H_0)) &= \frac{1}{2\pi i} \int_0^\infty f(\lambda) \textup{Tr}\left(S(\lambda)^*S'(\lambda) \right) \d \lambda  + \sum_{k=1}^K f(\lambda_k) M(\lambda_k) + f(0) \left( -N-\xi(0+) \right).
\end{align*}

Combining the results of Lemmas \ref{lem:guillope-det-S} and \ref{lem:det-limits} with Theorem \ref{thm:ssf-properties} we can determine the high-energy behaviour of the spectral shift function.

\begin{lemma}\label{lem:ssf-high-1}
Suppose that $\rho$ satisfies Assumption \ref{ass:best-ass}. If $n = 1$, let $p \geq 1$, if $n = 2,3$ let $p \geq 2$ and if $n \geq 4$ let $p \geq n$. Then we have the limit
\begin{align*}
\lim_{\lambda \to \infty} \left( -2\pi i \xi(\lambda) + \sum_{\ell  = 1}^{p-1} \frac{(-1)^\ell}{\ell} \textup{Tr} \left( (q_1 R_0(\lambda+i0) q_2)^\ell - (q_1 R_0(\lambda - i0) q_2 )^\ell \right) \right) &= 2\pi m
\end{align*}
for some $m \in \Z$.
\end{lemma}
\begin{proof}
Using Lemma \ref{lem:guillope-det-S} and Theorem \ref{thm:ssf-properties} we have
\begin{align*}
&\frac{\textup{Det}_p\left( \textup{Id}+q_1 R_0(\lambda-i0) q_2 \right)}{\textup{Det}_p\left( \textup{Id}+q_1 R_0(\lambda+i0) q_2 \right)} \\
&= \textup{Det}(S(\lambda)) \exp \left(\sum_{\ell  = 1}^{p-1} \frac{(-1)^\ell}{\ell} \textup{Tr} \left( (q_1 R_0(\lambda+i0) q_2)^\ell - (q_1 R_0(\lambda - i0) q_2 )^\ell \right) \right) \\
&= \exp\left(-2\pi i \xi(\lambda)+ \sum_{\ell  = 1}^{p-1} \frac{(-1)^\ell}{\ell} \textup{Tr} \left( (q_1 R_0(\lambda+i0) q_2)^\ell - (q_1 R_0(\lambda - i0) q_2 )^\ell \right) \right).
\end{align*}
By Lemma \ref{lem:det-limits} we have 
\begin{align*}
\lim_{\lambda \to \infty} \frac{\textup{Det}_p\left( \textup{Id}+q_1 R_0(\lambda-i0) q_2 \right)}{\textup{Det}_p\left( \textup{Id}+q_1 R_0(\lambda+i0) q_2 \right)}  &= 1,
\end{align*}
from which the result follows.
\end{proof}
\begin{rmk}\label{rmk:convention}
Taking the convention $m = 0$ in Lemma \ref{lem:ssf-high-1} determines a branch of the logarithm. We take this convention for the rest of this article.
\end{rmk}

In Section \ref{sec:traces} we will determine a more explicit expression for the high-energy behaviour of the spectral shift function using the traces in Lemma \ref{lem:ssf-high-1}. The fact that the spectral shift function does not converge as $\lambda \to \infty$ was known in dimension $n = 3$ as early as Buslaev \cite{buslaev62} (see also Newton \cite{newton77}, Boll\'{e} and Osborn \cite{bolle77} and Dreyfus \cite{dreyfus78, dreyfus78ii}) and in higher dimensions in the work of Guillop\'{e} \cite{guillope81, guillope85}. Precise high-energy asymptotics for the spectral shift function and its derivatives are well-known in the literature, see for example \cite{guillope81, popov82, robert94}.

\subsection{The form of the wave operator}

Define the operator
\begin{align}\label{eq:wave-op-form}
\varphi(D_n) &= \frac12 \left( \textup{Id}+\tanh{(\pi D_n)} - i A \cosh{(\pi D_n)}^{-1} \right),
\end{align}
where $A=  \textup{Id}$ if $n \geq 2$ and $A$ is a self-adjoint involution which commutes with $D_1$ if $n = 1$.

The following result is the culmination of a number of works, \cite[Theorem 5]{kellendonk08} in dimension $n = 1$, \cite[Theorem 1.3]{richard21} and \cite[Lemma 4.3]{ANRR} in dimension $n = 2$, \cite[Theorem 1.1]{richard13} in dimension $n = 3$, \cite[Theorem 3.1]{AR23} and \cite[Theorem 3.1]{AR23-4D} in dimension $n = 4$ and \cite[Theorem 3.1]{AR23} in dimension $n \geq 5$.

\begin{thm}\label{thm:wave-op-form}
Suppose that $V$ satisfies Assumption \ref{ass:best-ass}. Then the wave operator $W_-$ satisfies
\begin{align}
W_- &= \left(\textup{Id}+\varphi(D_n)(S-\textup{Id})\right) W_{res}+K,
\end{align}
where $K$ is a compact operator, $W_{res} = \textup{Id}$ for $n \neq 2,4$ and $W_{res}$ is a Fredholm operator depending on the existence of $s$-resonances in dimension $n = 4$ and on the existence of $p$-resonances in dimension $n = 2$. In particular, for $n = 4$ we have $\textup{Index}(W_{res})$ is equal to the number of $s$-resonances and for $n = 2$ we have $\textup{Index}(W_{res})$ is equal to the number of $p$-resonances.
\end{thm}
The explicit form of $W_{res}$ is not important for the analysis we present, and has been discussed thoroughly in dimension $n = 2$ in \cite{ANRR} and in dimension $n = 4$ in \cite{AR23-4D}.


\section{Generalised wave operators and index pairings}\label{sec:pairing}

Motivated by the form of the wave operator in Theorem \ref{thm:wave-op-form} we define a class of Fredholm operators, parametrised by particular unitaries, which allow us to reinterpret Levinson's theorem as an index pairing, as well as accounting for the contribution of resonances in dimensions $n = 1,3$. Let $D_n$ be as in Equation \eqref{eq:dilation-defn} and let $A \in \B(\H)$ be a self-adjoint involution which commutes with $D_n$. Define the operator $\varphi(D_n)$ by Equation \eqref{eq:wave-op-form}. 

If $T \in \B(\H)$ commutes with $H_0$, there exists a family $\{ T(\lambda) \}_{\lambda \in \R^+}$ of operators in $\B(\mathcal{P})$ such that $[F_0 T F_0^* f](\lambda,\omega) = T(\lambda) f(\lambda,\omega)$ for all $f \in \H_{spec}$, $\lambda \in \R^+$ and $\omega \in \Sf^{n-1}$. We call $T(\cdot)$ the matrix of $T$.

\begin{defn}
We say that a unitary $U \in \B(\H)$ is admissable if $[U,H_0] = 0$, the matrix of $U$ is norm continuous in $\lambda \in \R^+$ and $U(\lambda) - \textup{Id} \in \mathcal{K}(\mathcal{P})$ for all $\lambda \in \R^+$. We say that $U$ is properly admissable if in addition we have $\displaystyle U(0) = \lim_{\lambda \to \infty} U(\lambda) = \textup{Id}$, where the limit is taken in the norm of $\B(\mathcal{P})$.
\end{defn}

\begin{lemma}\label{lem:unitary-class}
Let $U \in \B(\H)$ be a properly admissable unitary. Then $U$ defines a $K$-theory class $[U] \in K_1\left(\left(C_0(\R^+) \otimes \mathcal{K}(\mathcal{P}) \right)^\sim \right)$.
\end{lemma}
\begin{proof}
Since $U-\textup{Id} \in C_0(\R^+) \otimes \mathcal{K}(\mathcal{P})$,  $U \in \left( C_0(\R^+) \otimes \mathcal{K}(\mathcal{P}) \right)^\sim$ and thus defines a $K$-theory class $[U] \in K_1\left(\left(C_0(\R^+) \otimes \mathcal{K}(\mathcal{P}) \right)^\sim \right)$.
\end{proof}

\begin{lemma}\label{lem:WU-fredholm}
Let $U \in \B(\H)$ be a properly admissable unitary. Then the operator $W_U \in \B(\H)$ defined by
\begin{align}\label{eq:WU-defn}
W_U &= \textup{Id}+ \varphi(D_n) (U-\textup{Id})
\end{align}
is a Fredholm operator.
\end{lemma}
\begin{proof}
We show that $W_U^*$ is an inverse for $W_U$ up to compacts. We compute that
\begin{align*}
W_U W_U^* &= \left( \textup{Id} + \varphi(D_n)(U-\textup{Id}) \right) \left( \textup{Id} + \left(U^*-\textup{Id} \right) \varphi(D_n)^* \right) \\
&= \textup{Id} + \varphi(D_n) (U-\textup{Id}) + (U^*-\textup{Id}) \varphi(D_n)^* + \varphi(D_n) (U-\textup{Id})(U^*-\textup{Id}) \varphi(D_n)^*.
\end{align*}
By \cite[Theorem 4.1.10]{amrein96} the commutators $[U^*, \varphi(D_n)^*]$ and $[U+U^*, \varphi(D_n)^*]$ are compact and thus we can write
\begin{align*}
W_U W_U^* &= \textup{Id} + \varphi(D_n) (U-\textup{Id}) +  \varphi(D_n)^* (U^*-\textup{Id}) + \varphi(D_n) \varphi(D_n)^*(U-\textup{Id})(U^*-\textup{Id}) +K_1,
\end{align*}
where $K_1$ is compact. By \cite[Lemma 3.3]{AR23} the operator $(\varphi(D_n)^* - \varphi(D_n)) (U^* - \textup{Id})$ is compact and so we have
\begin{align*}
W_U W_U^* &= \textup{Id} + \varphi(D_n) (U-\textup{Id}) +  \varphi(D_n) (U^*-\textup{Id}) + \varphi(D_n) \varphi(D_n)^*(U-\textup{Id})(U^*-\textup{Id}) +K_2,
\end{align*}
where $K_2$ is compact. Using the identity $\cosh{(\cdot)}^{-2}+\tanh{(\cdot)}^2 = 1$ and the fact that $A$ is a self-adjoint involution commuting with all functions of $D_n$ we find
\begin{align*}
\varphi(D_n)\varphi(D_n)^* &= \frac12 \left( \textup{Id} + \tanh{(\pi D_n)}\right).
\end{align*}
A further application of \cite[Lemma 3.3]{AR23} then gives that $\left(\varphi(D_n) \varphi(D_n)^*-\varphi(D_n) \right) (U-\textup{Id})$ is a compact operator and so we have
\begin{align*}
W_U W_U^* &= \textup{Id} + \varphi(D_n) (U-\textup{Id}) +  \varphi(D_n) (U^*-\textup{Id}) + \varphi(D_n) (U-\textup{Id})(U^*-\textup{Id}) +K_3 \\
&= \textup{Id}+K_3,
\end{align*}
where $K_3$ is a compact operator. A similar calculation shows that $W_U^*$ is a right inverse for $W_U$ up to compacts and so $W_U$ defines a Fredholm operator.
\end{proof}

Since the composition of properly admissable unitaries is itself a properly admissable unitary, we have a natural product rule for operators of the form $W_U$.

\begin{lemma}\label{lem:prod-rule}
Let $U_1, U_2 \in \B(\H)$ be properly admissable unitaries. For $j = 1,2$, define operators $W_{U_j}$ by Equation \eqref{eq:WU-defn}. Then we have $W_{U_1}W_{U_2} = W_{U_1U_2}$, up to a compact operator.
\end{lemma}
\begin{proof}
We compute the product
\begin{align*}
W_{U_1} W_{U_2} &= \textup{Id}+ \varphi(D_n)(U_1+U_2-2\textup{Id}) + \varphi(D_n) (U_1-\textup{Id}) \varphi(D_n) (U_2-\textup{Id}).
\end{align*}
Noting by \cite[Theorem 4.1.10]{amrein96} that the commutator $[U_1, \varphi(D_n)]$ is compact we have
\begin{align*}
W_{U_1} W_{U_2} &= \textup{Id}+ \varphi(D_n)(U_1+U_2-2\textup{Id}) + \varphi(D_n)^2 (U_1-\textup{Id})(U_2-\textup{Id}) + K_1
\end{align*}
for a compact operator $K_1$. Using, as in the proof of Lemma \ref{lem:WU-fredholm}, that the operator given by $\left(\varphi(D_n) \varphi(D_n)^*-\varphi(D_n) \right) (U_1-\textup{Id})$ is compact, we find
\begin{align*}
W_{U_1} W_{U_2} &= \textup{Id}+ \varphi(D_n)(U_1+U_2-2\textup{Id}) + \varphi(D_n) (U_1-\textup{Id})(U_2-\textup{Id}) + K_2 \\
&= \textup{Id} + \varphi(D_n) \left( U_1 U_2 - \textup{Id} \right) + K_2 \\
&= W_{U_1 U_2} + K_2,
\end{align*}
for a compact operator $K_2$.
\end{proof}

In fact Lemma \ref{lem:prod-rule} can be extended to show that we have a partially defined product rule.

\begin{lemma}\label{lem:partial-prod-rule}
Let $U_1, U_2 \in \B(\H)$ be admissable unitary operators with $U_1(0) = U_2(0)$ and $\displaystyle \lim_{\lambda \to \infty} U_1(\lambda) = \lim_{\lambda \to \infty} U_2(\lambda)$, where the limits are taken in the norm of $\B(\mathcal{P})$.  For $j = 1,2$, define operators $W_{U_j}$ by Equation \eqref{eq:WU-defn}. Then we have $W_{U_1} = W_{U_1U_2^*} W_{U_2}$, up to a compact operator.
\end{lemma}
\begin{proof}
Note that $U_1 U_2^*$ is a properly admissable unitary by construction. Hence we can apply (the proof of) Lemma \ref{lem:prod-rule} to see that
\begin{align*}
W_{U_1 U_2^*} W_{U_2} & = W_{U_1}+K
\end{align*}
for some compact operator $K$. 
\end{proof}

The following useful property shows that even if neither of the unitaries $U_1$ and $U_2$ are properly admissable, we may still deduce the Fredholm property of $W_{U_2}$ given that of $W_{U_1}$ and that both are close enough in some sense.

\begin{cor}\label{cor:12-Fredholm}
Let $U_1, U_2 \in \B(\H)$ be admissable unitary operators  with $U_1(0) = U_2(0)$ and $\displaystyle \lim_{\lambda \to \infty} U_1(\lambda) = \lim_{\lambda \to \infty} U_2(\lambda)$, where the limits are taken in the norm of $\B(\mathcal{P})$. For $j = 1,2$, define operators $W_{U_j}$ by Equation \eqref{eq:WU-defn}. Then if one of the operators $W_{U_j}$ is Fredholm, so is the other.
\end{cor}

\begin{proof}
Suppose, without loss of generality, that $W_{U_2}$ is Fredholm. Then since the unitary $U_1 U_2^*$ is properly admissable, we have $W_{U_1 U_2^*}$ is Fredholm by Lemma \ref{lem:WU-fredholm}. Since the composition of Fredholm operators is also Fredholm, Lemma \ref{lem:partial-prod-rule} shows that $W_{U_1}$ is Fredholm also.
\end{proof}

\subsection{The index pairing}

We now show how the form of the wave operator implies, generically, that the number of bound states can be computed as an index pairing between $K$-theory and $K$-homology, see \cite{CGRS2, HR}.

We begin by recalling the definition of a spectral triple.

\begin{defn}
An odd spectral triple $(\mathcal{A}, \H, \mathcal{D})$ is given by a Hilbert space $\H$, a dense \newline $*$-subalgebra $\mathcal{A} \subset \B(\H)$ acting on $\H$, and a densely defined unbounded self-adjoint operator $\mathcal{D}$ such that
\begin{enumerate}
\item $a \cdot \textup{Dom}(\mathcal{D}) \subset \textup{Dom}( \mathcal{D})$ for all $a \in \mathcal{A}$, so that $\d a := [\mathcal{D},a ]$ is densely defined. Moreover $\d a$ extends to a bounded operator for all $a \in \mathcal{A}$;
\item $a(\textup{Id}+\mathcal{D}^2)^{-\frac12} \in \mathcal{K}(\H)$ for all $a \in \mathcal{A}$.
\end{enumerate}
\end{defn}

The following is \cite[Corollary 4.4]{AR23}, providing the spectral triple of interest to us for scattering purposes.
\begin{lemma}\label{lem:spec-trip}
The data $\left( C_c^\infty(\R^+) \otimes \mathcal{K}(\mathcal{P}), L^2(\R^+) \otimes \mathcal{P}, D_+ \otimes \textup{Id}\right)$ defines an odd spectral triple, and so a class $[D_+]$ in odd $K$-homology.
\end{lemma}

Lemma \ref{lem:unitary-class} and Lemma \ref{lem:spec-trip} tell us that for a properly admissable unitary $U$ we can pair the classes $[D_+]$ and $[U]$ to obtain an integer (see \cite[Section 8.7]{HR} for details). The following can be proved in an identical manner to \cite[Theorem 4.7]{AR23}.

\begin{thm}\label{thm:pairing}
Let $U \in \B(\H)$ be a properly admissable unitary and define $W_U$ as in Equation \eqref{eq:WU-defn}. Let $[D_+]$ be defined by Lemma \ref{lem:spec-trip}. Then
\begin{align*}
\langle [U], [D_+] \rangle &= - \textup{Index}(W_U).
\end{align*}
\end{thm}

Generically, we have that $S(0) = \textup{Id}$ and thus we obtain, as in \cite[Theorem 4.7]{AR23} the following pairing for scattering operators.

\begin{lemma}\label{lem:pairing}
Let $H = H_0+V$ be such that the wave operators exist, are complete and are of the form of Equation \eqref{eq:WU-defn}. Let $S$ be the corresponding scattering operator. If $S(0) = \textup{Id}$ then we have the pairing
\begin{align*}
\langle [S], [D_+] \rangle &= - \textup{Index}(W_-) = N,
\end{align*}
the number of bound states for $H$.
\end{lemma}

We note that the $S(0) \neq \textup{Id}$ case cannot immediately be handled by the simple pairing desribed in Lemma \ref{lem:pairing} and we address this in Section \ref{sec:corrections}. In the presence of $p$-resonances in dimension $n =2$ and $s$-resonances in dimension $n = 4$ the wave operator is not of the form of Equation \eqref{eq:WU-defn} and thus requires some modification, resulting in the following.

\begin{lemma}
Suppose that $V$ satisfies Assumption \ref{ass:best-ass},  $n = 2,4$ and let $S$ be the corresponding scattering operator. Let $W_{res}$ be as in Theorem \ref{thm:wave-op-form}. Then we have the pairing
\begin{align*}
\langle [D_+], [S] \rangle &= - \textup{Index}(W_- W_{res}^*) = N + \textup{Index}(W_{res}),
\end{align*}
where $\textup{Index}(W_{res})$ is the number of $p$-resonances if $n=2$ and the number of $s$-resonances if $n = 4$.
\end{lemma}


\section{Trace class properties of the scattering matrix}\label{sec:traces}\label{sec:traces}

In this section we describe the high-energy behaviour of the scattering operator in the trace norm using the expression for $\textup{Det}(S(\lambda))$ of Lemma \ref{lem:guillope-det-S}. As a result we determine, using pseudodifferential expansions of the resolvent and the limiting absorption principle, the leading order high-energy asymptotics of the spectral shift function $\xi$ and its derivative. These leading order asymptotics are related to the heat kernel expansion as $t \to 0^+$ of the trace of the difference $\e^{-tH}-\e^{-tH_0}$.

\begin{lemma}\label{lem:LA-traces}
Suppose that $g_1, g_2: \R^n \to \C$ are compactly supported with $g = g_1 g_2$. For $\lambda > 0$ define the operator $B(\lambda) \in \B(\H)$ by $B(\lambda) = g_1 \left(R_0(\lambda+i0) - R_0(\lambda-i0) \right) g_2$. Then $B(\lambda) \in \mathcal{L}^1(\H)$ and 
\begin{align*}
\textup{Tr}\left( B(\lambda) \right) &= \frac{(2\pi i) \lambda^{\frac{n-2}{2}} \textup{Vol}(\Sf^{n-1})}{2(2\pi)^n} \int_{\R^n} g(x) \, \d x.
\end{align*}
Furthermore, for $t > \frac12$ and $\lambda > 0$ we have the equality $R_0(\lambda+i0)-R_0(\lambda-i0) = 2 \pi i \Gamma_0(\lambda)^* \Gamma_0(\lambda)$ as operators in $\B(H^{0,t}, H^{0,-t})$.
\end{lemma}
\begin{proof}
Fix $t > \frac12$ and $\lambda > 0$. As an operator in $\B(H^{0,t}, H^{0,-t})$ we have (see \cite[Equation 4.3]{agmon75} and \cite[Equation 15]{alsholm71}) that $A(\lambda) = R_0(\lambda+i0)-R_0(\lambda-i0)$ is an integral operator with integral kernel
\begin{align*}
A(\lambda,x,y) &= \frac{(2\pi i)}{2(2\pi)^n} \lambda^{\frac{n-2}{2}} \int_{\Sf^{n-1}} \e^{i\lambda^\frac12 \langle \omega, x-y \rangle} \, \d \omega,
\end{align*}
for $x,y \in \R^n$. Direct computation shows that the integral kernel of $(2\pi i) \Gamma_0(\lambda)^*\Gamma_0(\lambda)$ is the same as that of $A(\lambda)$.

That $B(\lambda)$ is trace class is proved in \cite[Lemma III.2]{guillope81} (see also the comments on the bottom of page $32$), however we provide an alternative proof. For any $t > \frac12$ direct computation shows that $g_1 \Gamma_0(\lambda)^* \in \mathcal{L}^2(\mathcal{P}, \H)$ and $\Gamma_0(\lambda) g_2 \in \mathcal{L}^2(\H,\mathcal{P})$ and thus  $B(\lambda) \in \mathcal{L}^1(\H)$.

We can then compute the trace of $B(\lambda)$ by integrating along the diagonal, which completes the proof.
\end{proof}
\begin{rmk}
The assumption of compact support on the potential $V$ is stronger than necessary to guarantee that the operator $B(\lambda)$ is trace class. It is sufficient that the multiplication operators corresponding to $g_1, g_2$ map $\H$ into $H^{0,t}$ for some $t > \frac12$. The computation of the trace of $B(\lambda)$ was first done by Buslaev \cite{buslaev62} (see also Newton \cite{newton77} and Boll\'{e} and Osborn \cite{bolle77}).
\end{rmk}
In fact we can differentiate $B(\lambda)$ arbitrarily in $\lambda$ in the trace norm.
\begin{lemma}\label{lem:LA-traces-higher}
Suppose that $g_1, g_2$ are compactly supported and $g = g_1 g_2$. For $\lambda > 0$ define the operator $B(\lambda) \in \mathcal{L}^1(\H)$ by $B(\lambda) = g_1 \left( R_0(\lambda+i0)-R_0(\lambda-i0) \right) g_2$. Then $B(\lambda)$ is differentiable in $\lambda$ in the norm of $\mathcal{L}^1(\H)$ and for all $\ell \in \N$ we have
\begin{align}\label{eq:deriv-B}
\frac{\d^{\ell-1}}{\d \lambda^{\ell-1}} B(\lambda) &=  (\ell-1)! g_1 \left( R_0(\lambda+i0)^\ell-R_0(\lambda-i0)^\ell \right) g_2.
\end{align}
In particular for $\lambda > 0$ and $\ell \in \N$ we have 
that  $ g_1 \left( R_0(\lambda+i0)^\ell-R_0(\lambda-i0)^\ell \right) g_2$ has for $x,y \in \R^n$ the integral kernel
\begin{align*}
\frac{(2\pi i) \Gamma \left( \frac{n}{2} \right) \lambda^{\frac{n}{2}-\ell}}{2 (\ell-1)!\Gamma \left( \frac{n}{2}+1-\ell \right) (2\pi)^n} g_1(x) g_2(y) \int_{\mathbb{S}^{n-1}} \e^{-i \lambda^\frac12 \langle \omega, y-x \rangle} \, \d \omega + \tilde{A}(\lambda, x,y)
\end{align*}
where $\tilde{A}(\lambda,\cdot,\cdot)$ vanishes on the diagonal and thus
\begin{align}\label{eq:higher-trace-formula}
\textup{Tr} \left( g_1 \left(R_0(\lambda+i0)^\ell - R_0(\lambda-i0)^\ell \right) g_2 \right) &= \frac{(2\pi i) \Gamma \left( \frac{n}{2} \right) \lambda^{\frac{n}{2}-\ell} \textup{Vol}(\Sf^{n-1})}{2 (\ell-1)!\Gamma \left( \frac{n}{2}+1-\ell \right) (2\pi)^n} \int_{\R^n} g(x) \, \d x.
\end{align}
\end{lemma}
\begin{proof}
That $B(\lambda)$ is differentiable is proved in \cite[Lemma 8.1.8]{yafaev10} and so we obtain Equation \eqref{eq:deriv-B}. Differentiating $(\ell-1)$ times the integral kernel for $B(\lambda)$ we obtain
\begin{align*}
\frac{\d^{\ell-1}}{\d \lambda^{\ell-1}} B(\lambda,x,y) &=  \frac{(2\pi i)}{2(2\pi)^n} g_1(x) g_2(y)  \sum_{j=0}^{\ell-1} \binom{\ell-1}{j} \left( \frac{\d^j}{\d \lambda^{j}} \lambda^{\frac{n-2}{2}} \right) \frac{\d^{\ell-1-j}}{\d \lambda^{\ell-1-j}} \int_{\Sf^{n-1}} \e^{i\lambda^\frac12 \langle \omega, x-y \rangle} \, \d \omega.
\end{align*}
By factoring as the product of two Hilbert-Schmidt operators as in Lemma \ref{lem:LA-traces}, each term in the sum is individually trace class. All terms except the $j = \ell-1$ term vanish on the diagonal since they contain an $\langle \omega, x-y \rangle^j$ term, so integrating over the diagonal gives Equation \eqref{eq:higher-trace-formula}.
\end{proof}
\begin{rmk}
The differentiability of $B(\lambda)$ and Equation \eqref{eq:deriv-B} can be checked directly at the level of kernels. Doing so demonstrates that again the assumption of compact support is stronger than necessary, it is enough that the multiplication operators $(1+|\cdot|)^\ell g_1, (1+|\cdot|)^\ell g_2$ map $\H$ into $H^{0,t}$ for some $t > \frac12$.
\end{rmk}

To evaluate some further traces, we need to be able to integrate polynomials over $\mathbb{S}^{n-1}$. We use the following result \cite{folland}.
\begin{lemma}\label{lem:folland}
Let $\alpha$ be a multi-index of length $n$ and let $P_\alpha:\R^n \to \C$ be given by $P_\alpha(x) = x^\alpha = x_1^{\alpha_1} \cdots x_n^{\alpha_n}$. Then
\begin{align*}
\int_{\Sf^{n-1}} P_\alpha(\omega)\, \d \omega &= \begin{cases} 0, \quad & \textup{ if some } \alpha_j \textup{ is odd,} \\
\frac{2 \Gamma\left( \frac{\alpha_1+1}{2} \right) \cdots \Gamma\left( \frac{\alpha_n+1}{2} \right)}{\Gamma\left( \frac{n+|\alpha|}{2} \right)}, \quad & \textup{ if all } \alpha_j \textup{ are even.}
\end{cases}
\end{align*}
\end{lemma}

\begin{lemma}\label{lem:heat-trace-diff}
Suppose that $\alpha$ is a multi-index of length $n$, $g \in C_c^\infty(\R^n)$ and let $X = g \partial^\alpha$, a differential operator of order $|\alpha|$. Then for $t > 0$ we have
\begin{align*}
\textup{Tr}\left( X \e^{-tH_0} \right) &= \frac{(-i)^{|\alpha|} \Gamma\left( \frac{\alpha_1+1}{2} \right) \cdots \Gamma \left( \frac{\alpha_n+1}{2} \right) t^{- \frac{n+|\alpha|}{2}}}{(2\pi)^n} \left( \int_{\R^n} g(x) \, \d x \right)
\end{align*}
if all $\alpha_j$ are even and $\textup{Tr}\left( X \e^{-tH_0} \right) = 0$ otherwise.
\end{lemma}
\begin{proof}
We compute the trace as
\begin{align*}
\textup{Tr}\left( X \e^{-tH_0} \right) &= \frac{(-i)^{|\alpha|}}{(2\pi)^n} \left( \int_{\R^n} g(x) \, \d x \right) \left( \int_{\R^n} y^\alpha \e^{-t|y|^2} \, \d \xi \right).
\end{align*}
If any of the components $\alpha_j$ is odd, then by changing to polar coordinates we find that $\textup{Tr}\left( X \e^{-tH_0} \right) = 0$ by Lemma \ref{lem:folland}. If all of the $\alpha_j$ are even we obtain
\begin{align*}
\textup{Tr}\left( X \e^{-tH_0} \right) &= \frac{2(-i)^{|\alpha|} \Gamma\left( \frac{\alpha_1+1}{2} \right) \cdots \Gamma \left( \frac{\alpha_n+1}{2} \right)}{(2\pi)^n\Gamma \left( \frac{n+|\alpha|}{2} \right)} \left( \int_{\R^n} g(x) \, \d x \right) \left( \int_0^\infty r^{|\alpha|+n-1} \e^{-tr^2} \, \d r \right) \\
&= \frac{(-i)^{|\alpha|} \Gamma\left( \frac{\alpha_1+1}{2} \right) \cdots \Gamma \left( \frac{\alpha_n+1}{2} \right) t^{- \frac{n+|\alpha|}{2}}}{(2\pi)^n} \left( \int_{\R^n} g(x) \, \d x \right),
\end{align*}
where we have used polar coordinates and Lemma \ref{lem:folland} to compute the spherical integral.
\end{proof}

\begin{lemma}\label{lem:resolvent-trace-diff}
Suppose that $\alpha$ is a multi-index of length $n$, $g_1, g_2 \in C_c^\infty(\R^n)$ and let $g =  g_1g_2$. Let $X = g_2 \partial^\alpha$ be a differential operator of order $|\alpha|$. Fix $\ell \in \N$ with $\ell \leq \frac{n+|\alpha|}{2}$ if $n$ is even. Then for $\lambda > 0$ we have
\begin{align*}
&\textup{Tr} \left( X \left( R_0(\lambda+i0)^\ell-R_0(\lambda-i0)^\ell \right) g_1 \right) \\
&= \frac{ (-i)^{|\alpha|}(2\pi i) \Gamma\left( \frac{\alpha_1+1}{2} \right) \cdots \Gamma\left( \frac{\alpha_n+1}{2} \right) \lambda^{\frac{n+|\alpha|}{2}-\ell}}{ (\ell-1)!\Gamma \left( \frac{n}{2}+1-\ell +\frac{|\alpha|}{2}\right) (2\pi)^n} \left(\int_{\R^n} g(x)\, \d x \right)
\end{align*}
if all $\alpha_j$ are even and $\textup{Tr} \left( X \left( R_0(\lambda+i0)^\ell-R_0(\lambda-i0)^\ell \right) g_1 \right) = 0$ otherwise. If $n$ is even and $\ell > \frac{n+|\alpha|}{2}$ then $\textup{Tr} \left( X \left( R_0(\lambda+i0)^\ell-R_0(\lambda-i0)^\ell \right) g_1 \right) = 0$ also.
\end{lemma}
\begin{proof}
Fix $\lambda > 0$. The integral kernel of $B(\lambda) = \left( R_0(\lambda+i0)-R_0(\lambda-i0) \right) g_1$ is given by
\begin{align*}
B(\lambda,x,y) &= \frac{(2\pi i)  \lambda^{\frac{n}{2}-1}}{2 (2\pi)^n} g_1(y) \int_{\mathbb{S}^{n-1}} \e^{-i \lambda^\frac12 \langle \omega, y-x \rangle} \, \d \omega.
\end{align*}
Letting $A(\lambda) =  X \left( R_0(\lambda+i0)-R_0(\lambda-i0) \right) g_1$ we see by Lemma \ref{lem:LA-traces} that $A(\lambda)$ has integral kernel
\begin{align*}
A(\lambda,x,y) &= \frac{(-i)^{|\alpha|}(2\pi i) \lambda^{\frac{n+|\alpha|}{2}-1}}{2 (2\pi)^n} g_2(x) g_1(y) \int_{\mathbb{S}^{n-1}} \omega_1^{\alpha_1} \cdots \omega_n^{\alpha_n} \e^{-i \lambda^\frac12 \langle \omega, y-x \rangle} \, \d \omega.
\end{align*}
Applying now the techniques of Lemma \ref{lem:LA-traces-higher} we differentiate $(\ell-1)$ times to find that the operator $A_\ell(\lambda) = X(R_0(\lambda+i0)^\ell-R_0(\lambda-i0)^\ell)$ has integral kernel
\begin{align*}
A_\ell(\lambda, x,y) &= \frac{(-i)^{|\alpha|}(2\pi i) \Gamma \left( \frac{n+|\alpha|}{2} \right) \lambda^{\frac{n+|\alpha|}{2}-1}}{2 \Gamma \left( \frac{n}{2}+1-\ell+\frac{|\alpha|}{2} \right) (2\pi)^n} g_2(x) g_1(y) \int_{\mathbb{S}^{n-1}} \omega_1^{\alpha_1} \cdots \omega_n^{\alpha_n} \e^{-i \lambda^\frac12 \langle \omega, y-x \rangle} \, \d \omega \\
&\quad + \tilde{A}(\lambda,x,y),
\end{align*}
where $\tilde{A}(\lambda,\cdot,\cdot)$ has vanishing trace.
Integrating over the diagonal gives that 
\begin{align*}
&\textup{Tr} \left( X \left( R_0(\lambda+i0)^\ell-R_0(\lambda-i0)^\ell \right) g_1 \right) \\
&= \frac{(-i)^{|\alpha|}(2\pi i) \Gamma \left( \frac{n+|\alpha|}{2} \right) \lambda^{\frac{n+|\alpha|}{2}-\ell}}{2(\ell-1)! \Gamma \left( \frac{n}{2}+1-\ell +\frac{|\alpha|}{2}\right) (2\pi)^n} \left(\int_{\R^n} g(x)\, \d x \right) \left(\int_{\mathbb{S}^{n-1}} \omega_1^{\alpha_1} \cdots \omega_n^{\alpha_n}  \, \d \omega\right).
\end{align*}
If some $\alpha_j$ is odd we find $\textup{Tr} \left( X \left( R_0(\lambda+i0)^\ell-R_0(\lambda-i0)^\ell \right) g_1 \right) = 0$ by Lemma \ref{lem:folland}. If all $\alpha_j$ are even we have
\begin{align*}
&\textup{Tr} \left( X \left( R_0(\lambda+i0)^\ell-R_0(\lambda-i0)^\ell \right) g_1 \right) \\
&= \frac{ (-i)^{|\alpha|}(2\pi i) \Gamma\left( \frac{\alpha_1+1}{2} \right) \cdots \Gamma\left( \frac{\alpha_n+1}{2} \right) \lambda^{\frac{n+|\alpha|}{2}-\ell}}{ (\ell-1)!\Gamma \left( \frac{n}{2}+1-\ell+\frac{|\alpha|}{2} \right) (2\pi)^n} \left(\int_{\R^n} g(x)\, \d x \right),
\end{align*}
again by Lemma \ref{lem:folland}.
\end{proof}

For the next statement, we introduce for $m \in \N \cup \{ 0 \}$ and $f \in C_c^\infty(\R^n)$ the notation $f^{(m)} = [H_0, [H_0, [\cdots, [H_0, f] \cdots ]] ]$ (where the expression has $m$ commutators). Note that $f^{(m)}$ is a differential operator of order $m$. 

\begin{lemma}\label{lem:pseudodifferential-resolvents}
Suppose that $q_1, q_2 \in C_c^\infty(\R^n)$ with $V = q_1 q_2$. Then for all $z \in \C \setminus \R$,  $\ell \in \N$ and $K \in \N \cup \{ 0 \}$ we have
\begin{align}\label{eq:psido-1}
(q_1 R_0(z) q_2 )^\ell &= q_1 \left( \sum_{|k|=0}^K (-1)^{|k|+1} C_{\ell-1}(k) V^{(k_1)} \cdots V^{(k_{\ell-1})} R_0(z)^{\ell+|k|}\right) q_2 + q_1 P_{K, \ell}(z)q_2,
\end{align}
where $P_{K,\ell}(z)$ is of order at most $-2\ell-K-1$, $k$ is a multi-index of length $(\ell-1)$ and
\begin{align*}
C_\ell(k) &= \frac{(|k|+\ell)!}{k_1 ! k_2 ! \cdots k_\ell ! (k_1+1) (k_1+k_2+2) \cdots (|k|+\ell)}.
\end{align*}
When $\ell = 1$ we have no remainder term.
For all $z \in \C \setminus \R$,  $M \in \N$ and $K  \in \N \cup\{ 0 \}$ we have
\begin{align}\label{eq:psido-2}
R(z) - R_0(z) &= \sum_{m=1}^M \left( \sum_{|k|=0}^K (-1)^{m+|k|} C_m(k) V^{(k_1)} \cdots V^{( k_m)} R_0(z)^{m+|k|+1} + P_{K,m}(z) \right) \\
&\quad + (-1)^{M+1} (R_0(z)V)^{M+1} R(z), \nonumber
\end{align}
where $P_{K,m}(z)$ has order (at most) $-2m-K-3$ and $k$ is a multi-index of length $m$.
\end{lemma}
\begin{proof}
Equation \eqref{eq:psido-1} follows from the pseudodifferential calculus of \cite[Lemma 6.11]{carey06}. For Equation \eqref{eq:psido-2}, we write
\begin{align*}
R(z) - R_0(z) &= \sum_{m=1}^M (-1)^m (R_0(z) V)^m R_0(z) + (-1)^{M+1} (R_0(z)V)^{M+1} R(z).
\end{align*}
Applying again \cite[Lemma 6.11]{carey06} we have
\begin{align*}
R(z) - R_0(z) &= \sum_{m=1}^M \left( \sum_{|k|=0}^K (-1)^{m+|k|} C_m(k) V^{(k_1)} \cdots V^{( k_m)} R_0(z)^{m+|k|+1} + P_{K,m}(z) \right) \\
&\quad  + (-1)^{M+1} (R_0(z)V)^{M+1} R(z),
\end{align*}
where $P_{K,m}(z)$ has order (at most) $-2m-K-3$.
\end{proof}

Since $V \in C_c^\infty(\R^n)$ and for $\ell \in \N$ and $k$ a multi-index of length $\ell$ we have $V^{(k_1)} \cdots V^{( k_\ell)}$ is a differential operator of order $|k|$ with smooth compactly supported coefficients, we can write
\begin{align}\label{eq:commutator-decomposition}
V^{(k_1)} \cdots V^{( k_\ell)} &= \sum_{|r|=0}^{|k|} g_{k,r} \partial^r,
\end{align}
with $r$ a multi-index with $n$ components and $g_{r,k} \in C_c^\infty(\R^n)$. With this notation, we obtain the following heat kernel expansion due to Colin de Verdi\`{e}re \cite{colin81} in odd dimensions (see also \cite[Theorem 3.64]{zworski19} for a general proof and \cite{banuelos96} for explicit computations of some coefficients). We provide a detailed proof using the pseudodifferential expansion of Lemma \ref{lem:pseudodifferential-resolvents} and the trace formulae of Lemma \ref{lem:heat-trace-diff}. As a result we obtain an expression for the coefficients which differs from \cite[Theorem 3.64]{zworski19} or \cite{banuelos96}, although is equivalent. This new expression is necessary for a comparison with the high-energy behaviour of the scattering operator, as we show in Theorem \ref{thm:high-energy-traces}.

\begin{prop}\label{prop:heat-kernel-expansion}
Suppose that $V \in C_c^\infty(\R^n)$. Then for any $J \in \N$ we have the expansion
\begin{align*}
\textup{Tr}\left( \e^{-tH}-\e^{-tH_0} \right) &= \sum_{j=1}^J a_j(n,V) t^{k- \frac{n}{2}} + E_J(t),
\end{align*}
where $E_J(t) = O\left(t^{J+1-\frac{n}{2}}\right)$ as $t \to 0^+$. For $j \in \N$ we define the set
\begin{align*}
Q_{M,K}(j) &= \bigg\{ (m,k,r) \in \{0,1,\dots, M \} \times \{0,1, \dots, K\}^m \times \{0,1, \dots, K \}^n : |r| \leq |k|, \\
& \quad \textup{all } r_j \textup{ are even and } m+|k|+1-\frac{|r|}{2} = j  \bigg\}.
\end{align*}
The constants $a_j(n, V)$ are given by
\begin{align*}
a_j(n,V) &= \sum_{(m,k,r) \in Q_{M,K}(j)}  \frac{(-i)^{|r|} C_m(k) (-1)^{m+|k|+1} \Gamma\left( \frac{r_1+1}{2} \right) \cdots \Gamma \left(\frac{r_n+1}{2} \right) }{(2\pi)^n(m+1)(m+|k|)!} \left( \int_{\R^n} V(x) g_{k,r}(x) \, \d x \right),
\end{align*}
where we have used the notation of Equation \eqref{eq:commutator-decomposition}.
\end{prop}
\begin{proof}
Fix $t > 0$. By \cite[Theorem 3.64]{zworski19} we have $\e^{-tH}-\e^{-tH_0} \in \mathcal{L^1}(\H)$. Choose $a > 0$ such that $a > |\lambda|$ for all $\lambda \in \sigma_p(H)$ and define the vertical line $\gamma_t = \{-a-\frac{1}{t}+iv: v \in \R\}$. We write
\begin{align*}
\e^{-tH}-\e^{-t H_0} &= \int_0^1 \frac{\d}{\d s} \e^{-t(H_0+sV)} \, \d s = -t \int_0^1 V \e^{-t(H_0+sV)} \, \d s.
\end{align*}
Cauchy's integral theorem tells us that for all $s \in [0,1]$ we have
\begin{align*}
\e^{-t(H_0+sV)} &= \frac{1}{2\pi i} \int_{\gamma_t} \e^{-tz} (z-H_0-sV)^{-1} \, \d z = - \frac{1}{2\pi i} \int_{\gamma_t} \e^{-tz} R_s(z) \, \d z,
\end{align*}
where we have introduced the notation $R_s(z) = (H_0+sV-z)^{-1}$. An application of Lemma \ref{lem:pseudodifferential-resolvents} gives us for all $M, K \in \N \cup \{0\}$ the expansion
\begin{align*}
R_s(z) &= \sum_{m=0}^M  \left(\sum_{|k|=0}^K (-1)^{m+|k|} s^m C_m(k) V^{(k_1)} \cdots V^{(k_m)} R_0(z)^{m+|k|+1} + s^m P_{K,m}(z) \right) \\
&\quad + (-1)^{M+1} s^{M+1} (R_0(z)V)^{M+1} R_s(z).
\end{align*}
Thus we write
\begin{align*}
&\textup{Tr}\left(V \e^{-t(H_0+sV)} \right) = -\textup{Tr}\bigg( \frac{1}{2\pi i} \int_{\gamma_t} \e^{-tz} R_s(z) \, \d z \bigg) \\
&= - \textup{Tr} \Bigg( \frac{1}{2\pi i} \int_{\gamma_t} \e^{-tz} \sum_{m=0}^M  \sum_{|k|=0}^K (-1)^{m+|k|} s^m C_m(k) VV^{(k_1)} \cdots V^{(k_m)}R_0(z)^{m+|k|+1} \, \d z \Bigg) \\
&-  \textup{Tr} \Bigg( \frac{1}{2\pi i} \int_{\gamma_t} \e^{-tz} \sum_{m=0}^M s^m VP_{K,m}(z)  \d z - \frac{1}{2\pi i} \int_{\gamma_t} \e^{-tz} (-1)^{M+1} s^{M+1} V(R_0(z) V)^{M+1} R_s(z)\d z \Bigg) \\
&:=  - \textup{Tr} \Bigg( \frac{1}{2\pi i} \int_{\gamma_t} \e^{-tz} \sum_{m=0}^M  \sum_{|k|=0}^K (-1)^{m+|k|} s^m C_m(k) VV^{(k_1)} \cdots V^{(k_m)}R_0(z)^{m+|k|+1} \, \d z \Bigg)  \\
&\quad + e_{M,K}(t,s).
\end{align*}
We now apply Cauchy's integral formula again to obtain
\begin{align*}
\frac{1}{2\pi i} \int_{\gamma_t} \e^{-tz} R_0(z)^{m+|k|+1} \, \d z &= (-1)^{m+|k|+1} \frac{1}{2\pi i} \int_{\gamma_t} \e^{-tz} (z-H_0)^{-m-|k|-1} \, \d z \\
&= \frac{(-1)^{m+|k|+1}}{(m+|k|)!} \frac{\d^{m+|k|}}{d z^{m+|k|}} \left( \e^{-tz} \right) \vert_{z = H_0} = - \frac{t^{m+|k|}}{(m+|k|)!} \e^{-tH_0}.
\end{align*}
Hence we find
\begin{align*}
\textup{Tr}\left(V \e^{-t(H_0+sV)} \right) &= \sum_{m=0}^M \sum_{|k|=0}^K (-1)^{m+|k|} s^m \frac{C_m(k)}{(m+|k|)!} t^{m+|k|} \textup{Tr} \left(VV^{(k_1)} \cdots V^{(k_m)} \e^{-tH_0} \right) \\
&\quad + e_{M,K}(t,s).
\end{align*}
Integrating out the $s$ variable we have
\begin{align*}
\textup{Tr} \left(\e^{-tH}-\e^{-tH_0} \right)  &= -t \int_0^1 \textup{Tr}\left( V \e^{-t(H_0+sV)}\right) \, \d s  \\
&= \sum_{m=0}^M \sum_{|k|=0}^K  \frac{(-1)^{m+|k|+1}C_m(k)}{(m+1)(m+|k|)!} t^{m+|k|+1}\textup{Tr} \left(VV^{(k_1)} \cdots V^{(k_m)} \e^{-tH_0} \right) \\
&\quad -t \int_0^1 e_{M,K}(t,s) \, \d s.
\end{align*}
We now use Lemma \ref{lem:heat-trace-diff} to write
\begin{align*}
&\sum_{m=0}^M \sum_{|k|=0}^K  \frac{(-1)^{m+|k|+1}C_m(k)}{(m+1)(m+|k|)!} t^{m+|k|+1}\textup{Tr} \left(VV^{(k_1)} \cdots V^{(k_m)} \e^{-tH_0} \right) \\
&= \sum_{m=0}^M \sum_{|k|=0}^K \mathop{\sum_{|r|=0}^{|k|}}_{r \textup{ even}} \frac{(-i)^{|r|}  (-1)^{\ell+m+|k|} C_m(k) \Gamma\left( \frac{r_1+1}{2} \right) \cdots \Gamma \left(\frac{r_n+1}{2} \right) t^{m+|k|+1-\frac{n+|r|}{2}}}{(m+1)(m+|k|)!(2\pi)^n} \\
&\times \left( \int_{\R^n} V(x) g_{k,r}(x) \, \d x \right),
\end{align*}
where the sum is over all multi-indices $r$ of length $n$ such that all $r_j$ are even. We now collect together powers of $t$. For $j \in \N$ we define the set
\begin{align*}
Q_{M,K}(j) &= \bigg\{ (m,k,r) \in \{0,1,\dots, M \} \times \{0,1, \dots, K\}^m \times \{0,1, \dots, K \}^n : |r| \leq |k|, \\
& \quad  \textup{all } r_j \textup{ are even and } m+|k|+1-\frac{|r|}{2} = j  \bigg\}
\end{align*}
and the coefficients
\begin{align*}
a_j(n,V) &= \sum_{(m,k,r) \in Q_{M,K}(j)}  \frac{(-i)^{|r|}(-1)^{m+|k|+1} C_m(k) \Gamma\left( \frac{r_1+1}{2} \right) \cdots \Gamma \left(\frac{r_n+1}{2} \right) }{(m+1)(m+|k|)! (2\pi)^n} \left( \int_{\R^n} V(x) g_{k,r}(x) \, \d x \right).
\end{align*}
The $a_j(n,V)$ allow us to write
\begin{align*}
\sum_{m=0}^M \sum_{|k|=0}^K  \frac{(-1)^{m+|k|+1}C_m(k)}{(m+1)(m+|k|)!} t^{m+|k|+1}\textup{Tr} \left(VV^{\{k_1\}} \cdots V^{\{k_m\}} \e^{-tH_0} \right) &= \sum_{j=1}^{M+1} a_j(n,V) t^{-\frac{n}{2}+j} + G(t),
\end{align*}
where $G(t) = O(t^{M+2-\frac{n}{2}})$ as $t \to 0^+$. Thus we have the expansion
\begin{align*}
\textup{Tr} \left(\e^{-tH}-\e^{-tH_0} \right)  &= \sum_{j=1}^{M+1} a_j(n,V) t^{-\frac{n}{2}+j} +G(t)-t \int_0^1 e_{M,K}(t,s).
\end{align*}
It remains to check the behaviour of the final remainder term, which consists of two types of terms. Recall that the operator $P_{K,m}(z)$ has order (at most) $-2m-K-3$. By \cite[Lemma 6.12]{carey06} there exists a constant $C$ (depending on $V$, $m$ and $K$ but not $z$) such that
\begin{align*}
\norm{R_0(z)^{-m-\frac{K}{2}-\frac12}q_2 P_{K,m}(z)} &\leq C.
\end{align*}
Choose $K = n-m$, large enough so that $q_1 R_0(z)^{m+\frac{K}{2}+\frac12}$ defines a trace-class operator. So we use H\"{o}lder's inequality for Schatten classes to find
\begin{align*}
\norm{VP_{K,m}(z)}_1 &\leq C \norm{q_1 R_0(z)^{m+\frac{K}{2}+\frac12}}_1.
\end{align*}
We can now estimate the remainder term via
\begin{align*}
\norm{\int_0^1 s^m  \int_{\gamma_t} \e^{-tz} P_{K,m}(z) \, \d z \, \d s}_1 &\leq \frac{C}{m+1} \int_\R \e^{at+1} \norm{q_1 R_0(z)^{m+\frac{K}{2}+\frac12}}_1 \, \d v \\
&\leq \tilde{C} \e^{at+1} \int_\R \left( \left(a+\frac{1}{t} \right)^2 + v^2 \right)^{-\frac{m}{2}-\frac{K}{4}-\frac{1}{4}+\frac{n}{4}} \, \d v.
\end{align*}
Make the substitution $v = \left(a+\frac{1}{t} \right) w$ to find
\begin{align*}
\norm{\int_0^1 s^m  \int_{\gamma_t} \e^{-tz} P_{K,m}(z) \, \d z \, \d s}_1 &\leq \tilde{C} \e^{at+1} \int_\R \left( \left(a+\frac{1}{t} \right)^2 + v^2 \right)^{-\frac{m}{2}-\frac{K}{4}-\frac{1}{4}+\frac{n}{4}} \, \d v \\
&= \tilde{C} \e^{at} \left(a+\frac{1}{t}\right)^{-m-\frac{K}{2}-\frac{1}{2}+\frac{n}{2}} \int_\R (1+w^2)^{-\frac{m}{2}-\frac{K}{4}-\frac{1}{4}+\frac{n}{4}} \, \d w \\
&\leq C \e^{at} t^{m+\frac{K}{2}+\frac12 - \frac{n}{2}}  \int_\R (1+w^2)^{-\frac{m}{2}-\frac{K}{4}-\frac{1}{4}+\frac{n}{4}} \, \d w \\
&= O\left(t^{m+\frac{K+1}{2}-\frac{n}{2}} \right)
\end{align*}
as $t \to 0+$. A similar estimate shows that
\begin{align*}
\norm{\int_0^1 s^{M+1} \int_{\gamma_t} V (R_0(z)V)^{M+1} R_s(z) \e^{-tz} \, \d z \, \d s}_1 &= O\left(t^{M+2-\frac{n}{2}} \right)
\end{align*}
as $t \to 0^+$. Taking $J = M+1$ completes the proof.
\end{proof}
\begin{rmk}
The heat kernel coefficients $a_j(n,V)$ are well-known and have been computed for small $j$ by many authors, see for example \cite{banuelos96} and \cite{colin81} in this context. The first few are
\begin{align*}
a_1(n,V) &= - \frac{\Gamma\left(\frac{n}{2} \right) \textup{Vol}(\Sf^{n-1})}{2(2\pi)^n} \int_{\R^n} V(x) \, \d x, \\
a_2(n,V) &= \frac{\Gamma\left( \frac{n}{2} \right) \textup{Vol}(\Sf^{n-1})}{4(2\pi)^n} \int_{\R^n} V(x)^2 \, \d x, \\
a_3(n,V) &= -\frac{\Gamma\left(\frac{n}{2} \right) \textup{Vol}(\Sf^{n-1})}{6(2\pi)^n} \int_{\R^n} \left( V(x)^3 + \frac12 |[\nabla V](x)|^2 \right) \, \d x.
\end{align*}
The expression for the $a_j(n,V)$ provided by Proposition \ref{prop:heat-kernel-expansion} provides a systematic way of computing these coefficients.
\end{rmk}
Truncating the heat kernel expansion of Proposition \ref{prop:heat-kernel-expansion} at $J = \lfloor \frac{n}{2} \rfloor$ we can write for $t > 0$ the expression
\begin{align}\label{eq:heat-truncated}
\textup{Tr} \left( \e^{-tH} - \e^{-tH_0} \right) &= \sum_{j=1}^{\lfloor \frac{n}{2} \rfloor} a_j(n,V) t^{j-\frac{n}{2}} + E_{\lfloor \frac{n}{2} \rfloor}(t)  \\
&= \sum_{j=1}^{\lfloor \frac{n-1}{2} \rfloor}  \frac{a_j(n,V)}{\Gamma\left(\frac{n}{2}-j \right)} \int_0^\infty \lambda^{\frac{n}{2}-j-1} \e^{-t \lambda} \, \d \lambda+ \frac{(-1)^n+1}{2} a_{\lfloor \frac{n}{2} \rfloor}(n,V) + E_{\lfloor \frac{n}{2} \rfloor}(t) \nonumber,
\end{align}
where we have separated out the constant term in even dimensions as
\begin{align}\label{eq:truncated}
\beta_n(V) &= \frac{(-1)^n+1}{2} a_{\lfloor \frac{n}{2} \rfloor}(n,V).
\end{align}

\begin{defn}\label{defn:high-energy-poly1}
We define the high-energy polynomial for $\xi'$ to be $p_n: (0,\infty) \to \C$ given for $\lambda \in (0,\infty)$ by
\begin{align*}
p_n(\lambda) &=  \sum_{j=1}^{\lfloor \frac{n-1}{2} \rfloor} c_j(n,V) \lambda^{\frac{n}{2}-j-1} := \sum_{j=1}^{\lfloor \frac{n-1}{2} \rfloor} \frac{(2\pi i) a_j(n,V)}{ \Gamma \left( \frac{n}{2}-j \right)} \lambda^{\frac{n}{2}-j-1}.
\end{align*}
By \cite[Theorem 1.2]{robert94} the high-energy polynomial $p_n$ is related to the the spectral shift function by $\displaystyle \lim_{\lambda \to \infty}\left( -2\pi i \xi'(\lambda) - p_n(\lambda) \right) = 0$.
\end{defn}
We can explicitly determine $p_n$ for small $n$ as $p_1 = p_2 = 0$, 
\begin{align*}
p_3(\lambda) &= - \frac{(2\pi i) \lambda^{-\frac12} \textup{Vol}(\Sf^2)}{4(2\pi)^3} \int_{\R^3} V(x) \, \d x = \frac{(2\pi i) \lambda^{-\frac12} a_1(3,V)}{\Gamma\left(\frac12 \right)}, \\
p_4(\lambda) &= - \frac{(2\pi i) \textup{Vol}(\Sf^3)}{2(2\pi)^4} \int_{\R^4} V(x) \, \d x = (2\pi i) a_1(4,V).
\end{align*}

We can use Proposition \ref{prop:heat-kernel-expansion} and the Birman-Kre\u{\i}n trace formula to analyse the integrability of the (derivative of) the spectral shift function on $\R^+$.

\begin{lemma}\label{lem:ssf-L1}
If $n = 1,2,3$ suppose that $V$ satisfies Assumption \ref{ass:best-ass}. If $n \geq 4$ suppose that $V \in C_c^\infty(\R^n)$. Then the function $\textup{Tr}\left(S(\cdot)^*S'(\cdot) \right) - p_n$ is integrable on $\R^+$. In particular, if $n = 1,2$ we have $\textup{Tr}\left(S(\cdot)^*S'(\cdot) \right) \in L^1(\R^+)$.
\end{lemma}
\begin{proof}
Since $\xi \vert_{(0,\infty)} \in C^1((0,\infty))$ by Theorem \ref{thm:ssf-properties} it suffices to check integrability in a neighbourhood of zero and a neighbourhood of infinity. The lowest power in the high-energy polynomial $p_n$ is $\lambda^{-\frac12}$ and thus $p_n$ is integrable in a neighbourhood of zero. That $\xi'$ is integrable in a neighbourhood of zero is the statement of \cite[Theorem 5.2]{jia12}. This can also be proved directly by using the resolvent expansions of \cite{jensen79, jensen80, jensen84, jensen01} and Equation \eqref{eq: scat matrix defn} to analyse the small $\lambda$ behaviour of $\textup{Tr}\left(S(\lambda)^*S'(\lambda)\right)$ as in \cite[Lemma 5.1]{christiansen19}.

For the claim regarding the integrability in a neighbourhood of infinity, we use the high-energy asymptotics of \cite[Theorem 1]{popov82}. The result is that as $\lambda \to \infty$ we have the expansion
\begin{align*}
\textup{Tr}\left(S(\lambda)^*S'(\lambda) \right) - p_n(\lambda) &= \sum_{j= \lfloor \frac{n-1}{2} \rfloor+1 }^\infty c_j(n,V) \lambda^{\frac{n}{2}-j-1},
\end{align*}
from which the integrability of $\textup{Tr}\left(S(\cdot)^*S'(\cdot) \right) - p_n$ in a neighbourhood of infinity follows (there is no coefficient of $\lambda^{-1}$ in even dimensions by \cite[Theorem 5.3]{jia12}).
\end{proof}

\begin{rmk}\label{rmk:heat-levinson}
In fact, at this point we can almost deduce Levinson's theorem from the Birman-Kre\u{\i}n trace formula and Lemma \ref{lem:ssf-L1} (see \cite[Theorem IV.5]{guillope81} and \cite[Section 5.B]{guillope85}). From Proposition \ref{prop:heat-kernel-expansion} we have the limit
\begin{align*}
0 &= \lim_{t \to 0^+} \bigg( \textup{Tr}\left(\e^{-tH}-\e^{-tH_0} \right) - \sum_{j=1}^{\lfloor \frac{n}{2} \rfloor} a_j(n,V) t^{k-\frac{n}{2}} \bigg).
\end{align*}
Using Equation \eqref{eq:heat-truncated} and the Birman-Kre\u{\i}n trace formula in the form of Lemma \ref{lem:birman-krein-ibp} we obtain
\begin{align*}
0 &= \lim_{t \to 0^+} \bigg( \frac{1}{2\pi i} \int_0^\infty \e^{-t\lambda} \textup{Tr} \left(S(\lambda)^* S'(\lambda) \right) \, \d \lambda + \sum_{k=1}^K \e^{-t\lambda_k} M(\lambda_k) \\
&\quad -N-\xi(0+) -\sum_{j=1}^{\lfloor \frac{n-1}{2} \rfloor} \frac{a_j(n,V)}{\Gamma\left(\frac{n}{2}-k \right)} \int_0^\infty \lambda^{\frac{n}{2}-k-1} \e^{-t \lambda} \, \d \lambda - \beta_n(V) \bigg) \\
&= \lim_{t \to 0^+} \Bigg( \frac{1}{2\pi i} \int_0^\infty \e^{-t\lambda} \left( \textup{Tr} \left(S(\lambda)^*S'(\lambda) \right) - \sum_{j=1}^{\lfloor \frac{n-1}{2} \rfloor} \frac{(2\pi i) a_j(n,V) \lambda^{\frac{n}{2}-k-1}}{\Gamma\left( \frac{n}{2}-k \right)} \right) \, \d \lambda \\
&\quad + \sum_{k=1}^K \e^{-t\lambda_k} M(\lambda_k)-N - \xi(0+) - \beta_n(V) \Bigg) \\
&= \lim_{t \to 0^+} \frac{1}{2\pi i} \int_0^\infty \e^{-t\lambda} \left( \textup{Tr} \left(S(\lambda)^*S'(\lambda) \right) - p_n(\lambda) \right) \, \d \lambda  - \xi(0+) - \beta_n(V).
\end{align*}
Here the constant $\beta_n(V)$ is as in Equation \eqref{eq:truncated}. An application of Lemma \ref{lem:ssf-L1} and the dominated convergence theorem then allows us to bring the limit as $t \to 0^+$ inside the integral to obtain
\begin{align*}
\xi(0+) &= \frac{1}{2\pi i} \int_0^\infty \left( \textup{Tr} \left(S(\lambda)^*S'(\lambda) \right) - p_n(\lambda) \right) \, \d \lambda -\beta_n(V).
\end{align*}
We see that the only missing information to obtain Levinson's theorem is the behaviour of the spectral shift function at zero in terms of eigenvalues and resonances. This information has been obtained directly in odd dimensions, see \cite{newton77} and \cite[Theorem 3.3]{guillope81}. The method of Remark \ref{rmk:heat-levinson} is used in odd dimensions by Colin de Verdi\`{e}re \cite{colin81}, Guillop\'{e} \cite{guillope81,guillope85} and Dyatlov and Zworski \cite[Theorem 3.66]{zworski19}. In Section \ref{sec:corrections} we take an alternative approach, proving Levinson's theorem directly from the index of the wave operator $W_-$ and as a consequence determining the behaviour of the spectral shift function at zero. 
\end{rmk}

We now analyse the high-energy behaviour of the (determinant of the) scattering operator using the limiting absorption principle and the pseudodifferential expansion of Lemma \ref{lem:pseudodifferential-resolvents}. To do so we define the following high-energy polynomial.
\begin{defn}\label{defn:high-energy-poly2}
We define the high-energy polynomial for $\xi$ to be $P_n: (0,\infty) \to \C$ given for $\lambda \in (0,\infty)$ by
\begin{align*}
P_n(\lambda) &=  2\pi i \beta_n(V)+\sum_{j=1}^{\lfloor \frac{n-1}{2} \rfloor} \frac{c_j(n,V)}{\frac{n}{2}-j} \lambda^{\frac{n}{2}-j} = \sum_{j=1}^{\lfloor \frac{n}{2} \rfloor} \frac{(2\pi i) a_j(n,V)}{\Gamma \left( \frac{n}{2}-j+1 \right)} \lambda^{\frac{n}{2}-j}.
\end{align*}
We note also that $P_n' =  p_n$, with $p_n$ the high-energy polynomial for $\xi'$ of Definition \ref{defn:high-energy-poly1}. 
\end{defn}
We can explicitly compute the lowest order polynomials, finding $P_1 = 0$,
\begin{align*}
P_2(\lambda) &= -\frac{ (2\pi i) \textup{Vol}(\Sf^1)}{2(2\pi)^2} \int_{\R^2} V(x) \, \d x = -\frac{2\pi i}{4\pi} \int_{\R^2} V(x) \, \d x = (2\pi i) \beta_2(V), \\
P_3(\lambda) &= -\frac{(2\pi i) \lambda^\frac12 \textup{Vol}(\Sf^2)}{2(2\pi)^3} \int_{\R^3} V(x) \, \d x = -\frac{(2\pi i)\lambda^\frac12}{4 \pi^2} \int_{\R^3} V(x)\, \d x =  \frac{(2\pi i) a_1(3,V) \lambda^\frac12}{\Gamma\left( \frac{3}{2} \right)}, \\
P_4(\lambda) &= - \frac{(2\pi i) \lambda \textup{Vol}(\Sf^3)}{2(2\pi)^4}\! \int_{\R^4} \!V(x)  \d x +\! \frac{(2\pi i) \textup{Vol}(\Sf^3)}{4(2\pi)^4}\! \int_{\R^4}\! V(x)^2 \d x =  (2\pi i) a_1(4,V) \lambda +2\pi i \beta_4(V).
\end{align*}

\begin{thm}\label{thm:high-energy-traces}
Suppose that $q_1,q_2 \in C_c^\infty(\R^n)$ with $V = q_1 q_2$. Then for all $\lambda > 0$ and $J \in \N$ we have
\begin{align}\label{eq:trace-expansion}
\sum_{\ell = 1}^J \left( \frac{(-1)^\ell}{\ell} \textup{Tr} \left( (q_1 R_0(\lambda+i0) q_2 )^\ell - (q_1 R_0(\lambda-i0))^\ell \right) \right) &= -\sum_{j=1}^{J} C_j(n,V) \lambda^{\frac{n}{2}-j}+E_{J}(\lambda),
\end{align}
where $E_{J}(\lambda) = O\left( \lambda^{\frac{n}{2}-J-3 } \right)$ as $\lambda \to \infty$ and $E_J$ is differentiable. If $n$ is even we have $C_j(n,V) = 0$ for all $j > \frac{n}{2}$. If $n$ is even and $j  \leq  \frac{n}{2}$ or $n$ is odd the coefficients $C_j(n,V)$ are given by
\begin{align}\label{eq:high-energy-to-heat}
C_j(n,V) &=  \frac{(2\pi i) a_j(n,V)}{\Gamma\left(\frac{n}{2}-j+1 \right)},
\end{align}
with the $a_j(n,V)$ the heat kernel coefficients of Proposition \ref{prop:heat-kernel-expansion}. Note also that for $n$ even we have $C_{\frac{n}{2}}(n,V) = 2\pi i \beta_n(V)$.
\end{thm}
\begin{proof}
For $\lambda > 0$, we have
\begin{align*}
\textup{Tr}\left( q_1 \left(R_0(\lambda+i0)\!-\!R_0(\lambda-i0) \right) q_2 \right) &= \frac{(2\pi i) \lambda^{\frac{n-2}{2}} \textup{Vol}(\Sf^{n-1})}{2(2\pi)^n}\! \int_{\R^n} V(x) \, \d x = -\frac{(2\pi i) a_1(n,V) \lambda^{\frac{n}{2}-1}}{ \Gamma \left( \frac{n}{2} \right)}.
\end{align*}
by Lemma \ref{lem:LA-traces}, with no need for any kind of expansion. We now consider the $\ell \geq 2$ terms in the sum.
For $z \in \C \setminus \R$ we use Lemma \ref{lem:pseudodifferential-resolvents} to obtain the expansion
\begin{align*}
&\sum_{\ell = 2}^L  \frac{(-1)^\ell}{\ell}  (q_1 R_0(z) q_2 )^\ell \\
&= \sum_{\ell = 2}^L \left( \frac{(-1)^\ell}{\ell} q_1 \left( \sum_{|k|=0}^{K} (-1)^{|k|+1} C_{\ell-1}(k) V^{(k_1)} \cdots V^{(k_{\ell-1})} R_0(z)^{\ell+|k|}\right) q_2 + \frac{(-1)^\ell}{\ell} q_1 P_{K, \ell}(z)q_2 \right),
\end{align*}
where $P_{K,\ell}(z)$ is of order (at most) $-2\ell-K-1$. By the limiting absorption principle, this equality extends to $z = \lambda \pm i0$ for $\lambda \in (0,\infty)$ and thus we find
\begin{align*}
&\sum_{\ell = 2}^L  \frac{(-1)^\ell}{\ell} \textup{Tr} \left( (q_1 (R_0(\lambda+i0)q_2)^\ell-(q_1 R_0(\lambda-i0)) q_2 )^\ell \right) \\
&= \sum_{\ell = 2}^L \textup{Tr} \Bigg(  q_1\! \left( \sum_{|k|=0}^{K} \frac{(-1)^{\ell+|k|+1}}{\ell} C_{\ell-1}(k) V^{(k_1)} \cdots V^{(k_{\ell-1})} \left(R_0(\lambda+i0)^{\ell+|k|}\!-\!R_0(\lambda-i0)^{\ell+|k|} \right)\right) \! q_2 \\
&\quad + \frac{(-1)^\ell}{\ell} q_1 \left(P_{K, \ell}(\lambda+i0) - P_{L, \ell}(\lambda-i0)\right)q_2 \Bigg).
\end{align*}
We now use the expansion of Equation \eqref{eq:commutator-decomposition} to write
\begin{align*}
&\sum_{\ell = 2}^L  \frac{(-1)^\ell}{\ell} \textup{Tr} \left( (q_1 R_0(\lambda+i0)-R_0(\lambda-i0)) q_2 )^\ell \right) \\
&= \sum_{\ell = 2}^L \textup{Tr} \Bigg(  q_1 \left( \sum_{|k|=0}^{K} \sum_{|r|=0}^{|k|} \frac{(-1)^{\ell+|k|+1}}{\ell} C_{\ell-1}(k) q_1 g_{k,r} \partial^r \left(R_0(\lambda+i0)^{\ell+|k|}-R_0(\lambda-i0)^{\ell+|k|} \right)\right) q_2 \\
&\quad + \frac{(-1)^\ell}{\ell} q_1 \left(P_{L, \ell}(\lambda+i0) - P_{L, \ell}(\lambda-i0)\right)q_2 \Bigg).
\end{align*}
Ignoring for a moment the remainder term we find using Lemma \ref{lem:LA-traces-higher} that
\begin{align*}
&\sum_{\ell = 2}^L \textup{Tr} \Bigg( \frac{(-1)^\ell}{\ell} q_1 \left( \sum_{|k|=0}^{K} \sum_{|r|=0}^{|k|} (-1)^{|k|+1} C_{\ell-1}(k) q_1 g_{k,r} \partial^r \left(R_0(\lambda+i0)^{\ell+|k|}-R_0(\lambda-i0)^{\ell+|k|} \right)\right) q_2 \bigg) \\
&= \sum_{\ell = 2}^L  \sum_{|k|=0}^{K} \sum_{|r|=0}^{|k|} \frac{(-1)^{\ell+|k|} (2\pi i)C_{\ell-1}(k)(-i)^{|r|}\Gamma \!\left( \frac{r_1+1}{2} \right) \cdots \Gamma \!\left( \frac{r_n+1}{2} \right)\! \lambda^{\frac{n+|r|}{2}-\ell-|k|}}{ \ell(\ell+|k|-1)! \Gamma \!\left( \frac{n}{2}+1-\ell-|k|+\frac{|r|}{2} \right)(2\pi)^n }\! \int_{\R^n} \!V(x) g_{k,r}(x) \, \d x .
\end{align*}
Lemma \ref{lem:LA-traces-higher} also gives that for $n$ even all terms with $\ell+|k|-\frac{|r|}{2} > \frac{n}{2}$ vanish. We now collect together the powers of $\lambda$. For $j \in \N$ we define the set
\begin{align*}
Q_{L,K}(j) &=  \bigg\{ (\ell, k,r) \in \{1, \dots, L\} \times \{ 0, \dots, K\}^\ell \times \{0, \dots, K \}^n : |r| \leq |k| \leq K,  \\
&\quad \textup{ all } r_j \textup{ are even, and } \ell + |k| - \frac{|r|}{2} = j \bigg\}
\end{align*}
So we write
\begin{align*}
&\sum_{\ell = 2}^L \textup{Tr} \Bigg( \frac{(-1)^\ell}{\ell} q_1 \left( \sum_{|k|=0}^{K} \sum_{|r|=0}^{|k|} (-1)^{|k|+1} C_{\ell-1}(k) q_1 g_{r,k} \partial^r \left(R_0(\lambda+i0)^{\ell+|k|}-R_0(\lambda-i0)^{\ell+|k|} \right)\right) q_2 \bigg) \\
&=-\sum_{j=2}^{J} C_j(n,V) \lambda^{\frac{n}{2}-j}, 
\end{align*}
where the coefficients $C_j(n,V)$ are given by
\begin{align*}
C_j(n,V) &= \sum_{(\ell, k,r) \in Q_{L,K}(j)} \!\frac{(-1)^{\ell+|k|}(2\pi i)C_{\ell-1}(k) (-i)^{|r|} \Gamma\left( \frac{r_1+1}{2} \right) \cdots \Gamma \left(\frac{r_n+1}{2} \right)}{\ell (\ell+|k|-1)! \Gamma\left( \frac{n}{2} + 1-\ell - |k|+\frac{|r|}{2} \right)(2\pi)^n}\! \int_{\R^n} \!V(x) g_{r,k}(x) \, \d x.
\end{align*}
Direct comparison with Proposition \ref{prop:heat-kernel-expansion} shows that we have the relation
\begin{align*}
C_j(n,V) &=  \frac{(2\pi i) a_j(n,V)}{\Gamma\left(\frac{n}{2}+1-j \right)}.
\end{align*}

We now return to the remainder term. Due to the form of the remainder terms $P_{K,\ell}(\lambda \pm i0)$ the difference $q_1 \left(P_{K,\ell}(\lambda+i0)-P_{K,\ell}(\lambda-i0)\right) q_2 $ is always trace-class, as we now show. The proof of \cite[Lemma 6.12]{carey06} shows that $P_{K,\ell}(\lambda \pm i0)$ is a linear combination of terms of the form
\begin{align*}
& \prod_{m=1}^M  A_m   R_0(\lambda \pm i0)^{\alpha_m},
\end{align*}
for some $M, \alpha_m \in \N$ and differential operators $A_m$ of order $a_m < 2\alpha_m$ with smooth compactly supported coefficients and
\begin{align*}
\sum_{m=1}^M \left(\alpha_m - \frac{a_m}{2} \right) &\geq -2\ell -L -1.
\end{align*}
Each $A_m$ can be factored as $f_m \tilde{A}_m g_m$ for $f_m, g_m \in C_c^\infty(\R^n)$ and $\tilde{A}_m$ also of order $a_m$. Taking the difference and factorising we find
\begin{align*}
& \left( \prod_{m=1}^M  A_m   R_0(\lambda + i0)^{\alpha_m} - \prod_{m=1}^M A_m  R_0(\lambda-i0)^{\alpha_m} \right)  \\
&=  \sum_{p=1}^M \left( \prod_{m < p} A_m R_0(\lambda-i0)^{\alpha_m} \right) \left( A_p \left(R_0(\lambda+i0)^{\alpha_p}-R_0(\lambda-i0)^{\alpha_p} \right) \right) \left( \prod_{m > p}  A_m R_0(\lambda+i0)^{\alpha_m} \right),
\end{align*}
each term of which is trace-class by Lemma \ref{lem:LA-traces-higher}. H\"{o}lder's inequality for the trace norm then shows that
\begin{align*}
&\norm{q_1 \left( \prod_{m=1}^M  A_m   R_0(\lambda + i0)^{\alpha_m} - \prod_{m=1}^M  A_m R_0(\lambda-i0)^{\alpha_m} \right) q_2}_1 \\
&\leq \sum_{p=1}^M \norm{q_1 f_1 \left(\prod_{m < p} \tilde{A}_m g_m  R_0(\lambda-i0)^{\alpha_m} f_{m+1} \right)  } \norm{ \tilde{A}_p g_p  \left( R_0(\lambda+i0)^{\alpha_p}-R_0(\lambda-i0)^{\alpha_p} \right)  f_{p+1} }_1 \\
&\quad \times  \norm{\tilde{A}_{p+1} g_{p+1}\left( \prod_{p < m \leq M}  R_0(\lambda+i0)^{\alpha_m}  f_{m+1} \right) q_2} ,
\end{align*}
where we use the convention $f_{M+1} = 1$.
By \cite[Theorem A.1]{agmon75} (see also \cite[Theorem 1]{murata84}) we have for any $\alpha > 0$ and differential operator $A$ of order $a < 2\alpha$ the estimate
\begin{align*}
\norm{q_1 A R_0(\lambda \pm i0)^\alpha q_2} &=O\left( \lambda^{-\frac12\alpha+\frac{a}{4}}\right)
\end{align*}
as $\lambda \to \infty$.
By Lemma \ref{lem:resolvent-trace-diff} we have the estimate
\begin{align*}
\norm{q_1 A \left(R_0(\lambda + i0)^\alpha-R_0(\lambda - i0)^\alpha \right)q_2}_1 &\leq C \lambda^{\frac{n}{2}-\alpha+\frac{a}{2}}.
\end{align*}
Combining these we obtain
\begin{align*}
&\norm{q_1 \left( \prod_{m=1}^M q_1 A_m q_2  R_0(\lambda + i0)^{\alpha_m} - \prod_{m=1}^M q_1 A_mq_2 R_0(\lambda-i0)^{\alpha_m} \right) q_2}_1 \\
&= C \sum_{p=1}^M \lambda^{\frac{n}{2}-\alpha_p+\frac{a_p}{2}} \left(\prod_{m \neq p} O\left(\lambda^{-\frac12\alpha_m+\frac{a_m}{4}} \right)\right) \leq O\left( \lambda^{\frac{n}{2}-\ell-\frac{L+1}{2}}\right)
\end{align*}
as $\lambda \to \infty$.
By varying $K$ we obtain Equation \eqref{eq:trace-expansion}. That $E_L$ is differentiable follows from the observation that the left-hand side of Equation \eqref{eq:trace-expansion} is differentiable, as is each power of $\lambda$ on the right-hand side.
\end{proof}
\begin{rmk}\label{rmk:choice}
Choosing $J \geq \lfloor \frac{n}{2} \rfloor$ in Theorem \ref{thm:high-energy-traces} gives 
\begin{align*}
\sum_{\ell = 1}^J \left( \frac{(-1)^\ell}{\ell} \textup{Tr} \left( (q_1 R_0(\lambda+i0) q_2 )^\ell - (q_1 R_0(\lambda-i0))^\ell \right) \right) &= -P_n(\lambda)-E(\lambda),
\end{align*}
with $E(\lambda) = O(\lambda^{-\frac12})$ as $\lambda \to \infty$ with $E$ differentiable. A comparison with Lemma \ref{lem:ssf-high-1} and using the convention of Remark \ref{rmk:convention} gives the relation
\begin{align*}
\lim_{\lambda \to \infty}\left( -2\pi i \xi(\lambda) - P_n(\lambda) \right) &= 0,
\end{align*}
justifying the terminology of $P_n$ being the high-energy polynomial for $\xi$.
\end{rmk}


\section{Low and high energy behaviour of the scattering matrix and Levinson's theorem}\label{sec:corrections}

Due to the generic failure of the equality $S(0) = \textup{Id}$ in dimensions $n = 1,3$ some adjustments are required to obtain an analogue of Theorem \ref{thm:pairing} in these cases via the use of an additional unitary operator. In higher dimensions, an additional unitary can be used to obtain a representative of the class $[S]$ with better behaviour at high-energy in the trace norm, from which we are able to deduce Levinson's theorem in each dimension.

To account for the low-energy behaviour of the scattering matrix, we introduce the following definition following \cite{carey23}.

\begin{defn}
Suppose that $V$ satisfies Assumption \ref{ass:best-ass} and let $S$ be the corresponding scattering operator. Let $\sigma$ be a once-continuously differentiable admissable unitary. Then $\textup{Det}(\sigma(\lambda)) \in \mathbb{T}$ exists for all $\lambda \in \R^+$ and we say that $\sigma$ is a low-energy correction for $S$ if $\sigma(0) = S(0)$,
\begin{align*}
\lim_{\lambda \to \infty} \sigma(\lambda) &= \textup{Id}
\end{align*}
when taken in $\B(\H)$, $W_\sigma$ is Fredholm and $\textup{Index}(W_\sigma) = 0$.
\end{defn}

We note that low-energy corrections are only required to account for the behaviour of the scattering matrix at zero in dimensions $n = 1,3$, since for all other dimensions we always have $S(0) = \textup{Id}$ by Theorem \ref{thm:scat-properties}. We will construct explicit low-energy corrections in these dimensions as needed.


To account for the high-energy behaviour of the scattering matrix in the trace norm, we introduce the following definition.
\begin{defn}
Suppose that $V$ satisfies Assumption \ref{ass:best-ass} and let $S$ be the corresponding scattering operator. Let $\beta$ be a once-continuously differentiable properly admissable unitary. Then $\textup{Det}(\beta(\lambda)) \in \mathbb{T}$ exists for all $\lambda \in \R^+$ and we say that $\beta$ is a high-energy correction for $S$ if
\begin{align*}
\lim_{\lambda \to \infty} \textup{Det}(S(\lambda) \beta(\lambda)) &= 1,
\end{align*}
and $\textup{Index}(W_\beta) = 0$.
\end{defn}

We note that high-energy corrections are not required in dimension $n = 1$, since in this case $\displaystyle \lim_{\lambda \to \infty} S(\lambda) = \textup{Id}$ in $\B(L^2(\Sf^0)) = M_2(\C) = \mathcal{L}^1(L^2(\Sf^0))$. We will construct explicit high-energy corrections as needed in all higher dimensions. 

We note the following key result, which follows immediately from \cite[Section 4.1]{kellendonk08}.

\begin{lemma}\label{lem:GK-index}
Let $U \in \B(L^2(\R^n))$ be a properly admissable unitary such that the matrix of $U$ is differentiable, $\R^+ \ni \lambda  \mapsto \textup{Det}(U(\lambda)) \in \mathbb{T}$ defines a loop and $|\textup{Tr}(U(\cdot)^* U'(\cdot)) | \in L^1(\R^+)$. Then we have
\begin{align}
\textup{Index}(W_U) &= \frac{1}{2\pi i} \int_0^\infty \textup{Tr}\left(U(\lambda)^*U'(\lambda) \right) \, \d \lambda.
\end{align}
\end{lemma}

\subsection{Low energy corrections in dimension $n = 1$ and $n = 3$}

As a direct result of Theorem \ref{thm:pairing} we have the following.

\begin{thm}
Suppose $n = 1,3$, $V$ satisfies Assumption \ref{ass:best-ass} and let $S$ be the corresponding scattering operator. Then for any low-energy correction $\sigma$ we have the pairing
\begin{align*}
\langle [D_+], [S\sigma^*] \rangle &= - \textup{Index}(W_{S\sigma^*}) = \textup{Index}(W_-) = N,
\end{align*}
where $N$ is the total number of bound states of $H$.
\end{thm}
\begin{proof}
Let $\sigma$ be a low-energy correction for $S$. Then the operator $W_{S\sigma^*}$ defines a properly admissable unitary and so by Theorem \ref{thm:pairing} we have the equality
\begin{align*}
\langle [D_+], [S\sigma^*] \rangle &= - \textup{Index}(W_{S\sigma^*}).
\end{align*}
Since $\sigma$ is a low-energy correction $\textup{Index}(W_{\sigma^*}) = 0$ and so $\textup{Index}(W_{S\sigma^*}) = \textup{Index}(W_S)$ by Lemma \ref{lem:partial-prod-rule}. Finally we have that $\textup{Index}(W_S) = \textup{Index}(W_-)$ by Theorem \ref{thm:wave-op-form} and $\textup{Index}(W_-) = -N$ by Proposition \ref{prop:index-wave-op}, which completes the proof.
\end{proof}

In dimension $n = 1$ we require two types of low-energy corrections.

\begin{defn}
We say a map $\sigma:\R^+ \to M_2(\C)$ is a generic correction if there exist differentiable functions $\theta:\R^+ \to \R$ and $f,g : \R^+ \to \C$ such that $\theta$ is increasing, $|f(\lambda)|^2+|g(\lambda)|^2 = 1$ for all $\lambda \in \R^+$, $\theta(\infty) \in 2\pi \Z$, $\theta(0) \in (2 \Z+1) \pi$, $\theta(\infty)-\theta(0) = \pi$, $-g(0) = f(\infty) = 1$ and $f(0) = g(\infty) = 0$ such that
\begin{align}\label{eq:1d-low-energy}
\sigma(\lambda) &= \begin{pmatrix} 
f(\lambda) & g(\lambda) \\
- \e^{i\theta(\lambda)} \overline{g(\lambda)} & \e^{i\theta(\lambda)} \overline{f(\lambda)}
\end{pmatrix}.
\end{align}
We say a map $\sigma:\R^+ \to M_2(\C)$ is a resonant correction if there exist differentiable functions $\theta:\R^+ \to \R$ and $f,g:\R^+ \to \C$ such that $\theta$ is increasing, $|f(\lambda)|^2+|g(\lambda)|^2 =  1$ for all $\lambda \in \R^+$, $\theta(\infty),\theta(0) \in 2\pi \Z$ and $\theta(\infty)-\theta(0) = 0$ with $f(0) = 2c_+c_-$, $g(0) = c_+^2-c_-^2$, $f(\infty) = 1$ and $g(\infty) = 0$ for some $c_+, c_- \in \R\setminus \{0\}$ with $c_+^2+c_-^2 = 1$ such that Equation \eqref{eq:1d-low-energy} holds.
\end{defn}

We note that the existence of low-energy corrections in dimension $n = 1$ is guaranteed by explicit choices of the functions $\theta,f,g$. For example in the generic case we can choose
\begin{align*}
\theta(\lambda) &= 2 \tan^{-1}(\lambda), \quad g(\lambda) =- \e^{-\lambda}, \quad \textup{ and } \quad f(\lambda) = \sqrt{1-\e^{-2\lambda}}.
\end{align*}

\begin{lemma}\label{lem:sigma-1d}
Suppose $n = 1$, $V$ satisfies Assumption \ref{ass:best-ass} and let $S$ be the corresponding scattering operator. If $H$ has no resonances, let $\sigma$ be a generic correction and if $H$ has a resonance let $\sigma$ be a resonant correction with $c_+,c_-$ determined by Theorem \ref{thm:scat-mat-zero}. Then $\sigma$ defines a low-energy correction for $S$.
\end{lemma}
\begin{proof}
We first make the observation that
\begin{align*}
\sigma(0) &= \begin{cases}
\begin{pmatrix} 0 & - 1 \\ -1 & 0 \end{pmatrix}, \quad & \textup{ if } \sigma \textup{ is a generic correction}, \\
\begin{pmatrix} 2c_+ c_- & c_+^2 - c_-^2 \\ c_-^2 - c_+^2 & 2 c_+ c_- \end{pmatrix}, \quad & \textup{ if } \sigma \textup{ is a resonant correction}.
\end{cases}
\end{align*}
Thus in both the generic and resonant cases we have that $\sigma(0) = S(0)$.  By construction we also have that
\begin{align*}
\lim_{\lambda \to \infty} \sigma(\lambda) &= \textup{Id}.
\end{align*}
Since $\sigma(0) = S(0)$ and $W_S$ is Fredholm, we find that $W_\sigma$ is Fredholm also by Lemma \ref{lem:partial-prod-rule}. It remains to check that $\textup{Index}(W_\sigma) = 0$. To do so, we use Gohberg-Kre\u{\i}n index theory in the form of \cite[Section 4.1]{kellendonk08}. In particular, we have that 
\begin{align*}
\textup{Index}(W_\sigma) &= \frac{1}{2\pi i} \int_0^\infty \frac{\frac{\d}{\d \lambda} \textup{Det}(\sigma(\lambda))}{\textup{Det}(\sigma(\lambda))} \, \d \lambda + \frac{1}{2\pi i} \int_\R \frac{\frac{\d}{\d s} \textup{Det}( \textup{Id} + \varphi(s)(\sigma(0)-\textup{Id}))}{\textup{Det}( \textup{Id} + \varphi(s)(\sigma(0)-\textup{Id}))} \, \d s.
\end{align*}
The first term is easily evaluated as
\begin{align*}
\frac{1}{2\pi i} \int_0^\infty \frac{\frac{\d}{\d \lambda} \textup{Det}(\sigma(\lambda))}{\textup{Det}(\sigma(\lambda))} \, \d \lambda &= \frac{1}{2\pi i} \int_0^\infty \theta'(\lambda) \, \d \lambda \\
&= \begin{cases}
\frac12, \quad & \textup{ if } \sigma \textup{ is a generic correction}, \\
0, \quad & \textup{ if } \sigma \textup{ is a resonant correction}.
\end{cases}
\end{align*}
The second integral has already been evaluated in \cite[Proposition 9]{kellendonk08} (using a different choice of basis) as
\begin{align*}
\int_\R \frac{\frac{\d}{\d s} \textup{Det}( \textup{Id} + \varphi(s)(\sigma(0)-\textup{Id}))}{\textup{Det}( \textup{Id} + \varphi(s)(\sigma(0)-\textup{Id}))} \, \d s &= -\frac12
\end{align*}
and thus $\textup{Index}(W_\sigma) = 0$.
\end{proof}

\begin{lemma}\label{lem:3d-low-energy}
Let $n = 3$ and suppose that $V$ satisfies Assumption \ref{ass:best-ass} with $S$ the corresponding scattering operator. Let $P_s$ be the projection onto the spherical harmonics of order zero in $L^2(\Sf^2)$ if there exists a resonance for $H = H_0+V$ and $P_s = 0$ otherwise. Let $\theta: \R^+ \to \R$ be an increasing differentiable function such that $\theta(0) = 0$ and $\theta(\infty) = \pi$. Then the unitary operator $\sigma \in \B(\H)$, defined for $\lambda \in \R^+$ by 
\begin{align*}
\sigma(\lambda) &= \textup{Id} - \left(1+\e^{i\theta(\lambda)} \right) P_s,
\end{align*}
defines a low-energy correction for $S$.
\end{lemma}
\begin{proof}
If $P_s = 0$ then $\sigma(\lambda) = \textup{Id}$ and the result is clear, so we suppose that there does exist a resonance. Since $P_s$ is a finite rank projection, we have that $\sigma(\lambda)-\textup{Id}$ is trace class for all $\lambda \in \R^+$. By construction we have $\sigma(0) = \textup{Id}-2P_s = S(0)$ and $\sigma(\infty) = \lim_{\lambda \to \infty} S(\lambda) = \textup{Id}$ (with the limit taken in $\B(\mathcal{P})$). Since $W_- = W_S$ (up to compacts) by Theorem \ref{thm:wave-op-form}, we have that $W_S$ is Fredholm and thus $W_\sigma$ is Fredholm by Corollary \ref{cor:12-Fredholm}. We can compute using Gohberg-Kre\u{\i}n theory in the form of \cite[Section 6]{kellendonk12} (see also \cite{GK} and \cite{richard16}) that
\begin{align*}
\textup{Index}(W_\sigma) &= \frac{1}{2\pi i} \int_0^\infty \frac{ \frac{\d}{\d \lambda} \textup{Det}(\sigma(\lambda))}{\textup{Det}(\sigma(\lambda))} \, \d \lambda + \frac{1}{2\pi i} \int_\R \frac{\frac{\d}{\d s} \textup{Det}\left(\textup{Id}+\varphi\left(-2s \right)(S(0)-\textup{Id})\right)}{ \textup{Det}\left(\textup{Id}+\varphi\left(-2s \right)(S(0)-\textup{Id})\right)} \, \d s.
\end{align*}
The first integral can be evaluated as
\begin{align*}
\frac{1}{2\pi i} \int_0^\infty \frac{\frac{\d}{\d \lambda} \textup{Det}(\sigma(\lambda))}{\textup{Det}(\sigma(\lambda))} \, \d \lambda &= \frac{1}{2\pi i} \int_0^\infty \frac{i \theta'(\lambda) \e^{i \theta(\lambda)}}{\e^{i\theta(\lambda)}} \, \d \lambda = \frac12.
\end{align*}
For the second integral, since $S(0) -\textup{Id} = - 2P_s$ and $P_s$ is a rank-one projection, we find that the operator $\varphi\left(-2s \right)(S(0)-\textup{Id})$ has a single eigenvalue given by $-2\varphi \left( -2s \right)$. Thus we find
\begin{align*}
\textup{Det}\left( \textup{Id}+\varphi\left(-2s \right)(S(0)-\textup{Id}) \right) &= 1 -2\varphi \left( -2s \right).
\end{align*}
Hence we can evaluate the $s$-integral as
\begin{align*}
\frac{1}{2\pi i} \int_\R \frac{\frac{\d}{\d s} \textup{Det}\left(\textup{Id}+\varphi\left(-2s \right)(S(0)-\textup{Id})\right)}{ \textup{Det}\left(\textup{Id}+\varphi\left(-2s \right)(S(0)-\textup{Id})\right)} \, \d s &= \frac{1}{2\pi i} \int_\R \frac{\frac{\d }{\d s} \left( 1 -2  \varphi \left( -2s \right) \right)}{1 -2  \varphi \left( - 2s \right)} \, \d s \\
&= \frac{1}{2\pi i} \int_\R \frac{\frac{\d}{\d s} \left(\tanh \left( 2\pi s \right) + i \cosh \left( 2\pi s \right)^{-1} \right)}{\tanh \left( 2\pi s\right) + i \cosh \left( 2 \pi s \right)^{-1}} \, \d s \\
&= -2 \int_\R \frac{\e^{2 \pi s}}{\e^{4 \pi s}+1} \, \d s = -\frac12.
\end{align*}
Combining these we find $\textup{Index}(W_\sigma) = 0$ so that $\sigma$ defines a low-energy correction for $S$.
\end{proof}

\subsection{High energy corrections in dimension $n \geq 2$}
We now construct high-energy corrections for dimension $n \geq 2$. We will use explicitly the trace relations of Section \ref{sec:traces} and thus we impose the stronger condition that $V \in C_c^\infty(\R^n)$. In dimension $n = 1$, no corrections are necessary as $\displaystyle \lim_{\lambda \to \infty}S(\lambda) = \textup{Id}$ in trace norm.
\begin{lemma}\label{lem:high-energy-nbig}
Suppose that $n \geq 2$ and $q_1, q_2 \in C_c^\infty(\R^n)$ with $q_1q_2 = V$ and corresponding scattering operator $S$ and fix $0 \neq \tilde{\chi} \in C_c^\infty(\R^+)$ with $\tilde{\chi}(\lambda) = 0$ for $\lambda \leq 1$. Define $\chi:\R^+ \to \R$ for $\lambda \in \R^+$ by
\begin{align*}
\chi(\lambda) &= \frac{\int_0^\lambda \tilde{\chi}(u) \, \d u}{\int_0^\infty \tilde{\chi}(u) \, \d u}.
\end{align*}
If $n = 2,3$ let $p = 2$ and if $n \geq 4$ let $p \geq n$. For $\lambda > 0$ define the self-adjoint operator $\tilde{A}(\lambda) \in \B(\mathcal{P})$ by
\begin{align*} 
\tilde{A}(\lambda) &= -2\pi \Gamma_0(\lambda) q_2 \left( \sum_{\ell = 1}^{p -1} \sum_{j=0}^{\ell - 1}  \frac{(-1)^\ell}{\ell} \left( q_1 R_0(\lambda+i0) q_2 \right)^j \left( q_1 R_0(\lambda-i0) q_2 \right)^{\ell-j-1} \right) q_1 \Gamma_0(\lambda)^*,
\end{align*}
For $\lambda \in \R^+$ let $A(\lambda) = \chi(\lambda^\frac12) \tilde{A}(\lambda)$. Then the unitary operator $\beta(\lambda) = \e^{i A(\lambda)}$ defines a high-energy correction for $S$.
\end{lemma}
\begin{proof}
By construction we have that $A(\lambda)$ is self-adjoint for all $\lambda \in \R^+$.
We have $A(0) = 0$ and the norm limit
\begin{align*}
\lim_{\lambda \to \infty} A(\lambda) &= 0
\end{align*}
by an application of the estimate \eqref{eq:estimate-ref} and \cite[Theorem 1]{murata84}. By Lemma \ref{lem:LA-traces} we have that $A(\lambda)$ is a trace-class operator and by cyclicity of the trace we have
\begin{align*}
\textup{Tr}(A(\lambda)) &= i \chi(\lambda^\frac12) \sum_{\ell = 1}^{p-1} \frac{(-1)^\ell}{\ell} \textup{Tr} \left( \left( q_1 R_0(\lambda+i0) q_2 \right)^\ell-\left( q_1 R_0(\lambda-i0) q_2 \right)^\ell \right).
\end{align*}
As a consequence of H\"{o}lder's inequality for Schatten ideals, $\beta(\lambda) - \textup{Id} \in \mathcal{L}^1(\mathcal{P})$ and so $\beta$ defines a properly admissable unitary. Thus $\textup{Det}(\beta(\lambda)) = \e^{i \textup{Tr}(A(\lambda))}$ and so by Lemma \ref{lem:guillope-det-S} we find that
\begin{align*}
&\textup{Det}(S(\lambda)) \textup{Det}(\beta(\lambda))\frac{\textup{Det}_p(\textup{Id}+q_1 R_0(\lambda+i0) q_2)}{\textup{Det}_p(\textup{Id}+q_1 R_0(\lambda-i0) q_2)} \\
 &=  \exp\left( \sum_{\ell = 1}^{p-1} \frac{(-1)^\ell}{\ell} \textup{Tr} \left( (q_1R_0(\lambda+i0)q_2)^\ell-(q_1 R_0(\lambda-i0) q_2)^\ell \right) + i\textup{Tr}(A(\lambda)) \right) \\
&=  \exp\left((1-\chi(\lambda^\frac12)) \sum_{\ell = 1}^{p-1} \frac{(-1)^\ell}{\ell} \textup{Tr} \left( (q_1R_0(\lambda+i0)q_2)^\ell-(q_1 R_0(\lambda-i0) q_2)^\ell \right) \right).
\end{align*}

An application of Theorem \ref{thm:high-energy-traces} and Lemma \ref{lem:det-limits} gives that
\begin{align*}
\lim_{\lambda \to \infty} \textup{Det}(S(\lambda) \beta(\lambda)) &= 1.
\end{align*}
To see that $\textup{Index}(W_\beta) = 0$, we consider for $t \in [0,1]$ the homotopy
\begin{align*}
A_t(\lambda) =  2 \pi \chi((1-t)\lambda^\frac12) \tilde{A}(\lambda).
\end{align*}
The path $A_t(\lambda)$ defines a norm-continuous path in $\B(\mathcal{P})$ from $A_0(\lambda) = A(\lambda)$ to $A_1(\lambda) = 0$. Defining the path $A_t = A_t(L) \in \B(\H)$ we obtain a norm continuous path in $\B(\H)$ from $A$ to $0$. To see this, fix $t_1, t_2 \in [0,1)$ and define $d = (\min\{ (1-t_1)^{-1}, (1-t_2)^{-1} \})^2 \geq 1$. Then for $g \in \H_{spec}$ we find
\begin{align*}
&\norm{(A_{t_2}(\cdot)-A_{t_1}(\cdot))g}_{\H_{spec}}^2 \\
&= \int_0^\infty \int_{\Sf^{n-1}} |[(A_{t_2}(\lambda)-A_{t_1}(\lambda))g(\lambda,\cdot)](\omega)|^2\, \d \omega \, \d \lambda \\
&= \int_d^\infty \int_{\Sf^{n-1}} |[(A_{t_2}(\lambda)-A_{t_1}(\lambda))g(\lambda,\cdot)](\omega)|^2\, \d \omega \, \d \lambda \\
&\leq \int_d^\infty |\chi(\lambda^\frac12(1-t_2))-\chi(\lambda^\frac12(1-t_1))|^2\int_{\Sf^{n-1}} |[\tilde{A}(\lambda)g(\lambda,\cdot)](\omega)|^2\, \d \omega \, \d \lambda \\
&\leq C |t_2-t_1|^2 \int_d^\infty \lambda \int_{\Sf^{n-1}} |[\tilde{A}(\lambda)g(\lambda,\cdot)](\omega)|^2\, \d \omega \, \d \lambda.
\end{align*}
It remains to estimate $|[\tilde{A}(\lambda) g(\lambda,\cdot)](\omega)|$. We recall from \cite[Theorem 1]{murata84} that for sufficiently large $\lambda$ we have
\begin{align*}
\norm{q_1 R_0(\lambda\pm i 0)^j q_2} &= O(\lambda^{-\frac{j}{2}})
\end{align*}
and so we may define
\begin{align*}
K &= \sup\left\{ 2\pi \sum_{\ell = 1}^{p-1} \sum_{j=0}^{\ell - 1} \frac{1}{\ell} \norm{(q_1 R_0(\lambda + i0) q_2)^j (q_1 R_0(\lambda-i0) q_2)^{\ell-j-1}}: \lambda \in [d,\infty) \right\}.
\end{align*} 
Using twice the estimate \eqref{eq:estimate-ref} we obtain
\begin{align*}
&\int_{\Sf^{n-1}} |[\tilde{A}(\lambda)g(\lambda,\cdot)](\omega)|^2 \, \d \omega = \norm{A(\lambda)g(\lambda,\cdot)}^2_{L^2(\Sf^{n-1})} \\
&\leq 2\pi \lambda^{-\frac12} \sum_{\ell = 1}^{p-1} \sum_{j=0}^{\ell - 1} \frac{1}{\ell} \norm{(q_1 R_0(\lambda + i0) q_2)^j (q_1 R_0(\lambda-i0) q_2)^{\ell-j-1} q_1 \Gamma_0(\lambda)^* g(\lambda,\cdot)}^2 \\
&\leq K \lambda^{-\frac12} \norm{q_1 \Gamma_0(\lambda)^* g(\lambda,\cdot)}^2 \leq K \lambda^{-1} \norm{g(\lambda,\cdot)}^2_{L^2(\Sf^{n-1})}.
\end{align*}
Combining these we find
\begin{align*}\norm{(A_{t_2}(\cdot)-A_{t_1}(\cdot))g}_{\H_{spec}}^2 &\leq C |t_2-t_1|^2 \int_d^\infty \lambda \int_{\Sf^{n-1}} |[\tilde{A}(\lambda)g(\lambda,\cdot)](\omega)|^2\, \d \omega \, \d \lambda \\
&\leq CK |t_2-t_1|^2 \int_d^\infty \norm{g(\lambda,\cdot)}^2_{L^2(\Sf^{n-1})} \, \d \lambda \\
&\leq CK|t_2-t_1|^2 \norm{g}^2_{\H_{spec}}.
\end{align*}
The case when either or both of $t_1, t_2$ is one follows from a similar calculation.

As a result, the path $\beta_t = \e^{iA_t}$ defines a norm continuous path of unitary operators in $\B(\H)$ from $\beta$ to $\textup{Id}$. Hence the path $W_{\beta_t}$ defines a norm continuous path in $\B(\H)$ from $W_\beta$ to $\textup{Id}$, along which the Fredholm index is constant and equal to zero.
\end{proof}

\subsection{Levinson's theorem and the spectral shift function at zero}
We now use the high-energy behaviour of the spectral shift function of Section \ref{sec:traces} to prove Levinson's theorem in all dimensions. As a corollary, we deduce the behaviour of the spectral shift function at zero in all dimensions.

\begin{prop}\label{prop:index-WS-nbig}
Suppose that $V \in C_c^\infty(\R^n)$ with corresponding scattering operator $S$. Then we have
\begin{align}
\textup{Index}(W_S) &= \frac{1}{2\pi i} \int_0^\infty \left( \textup{Tr}\left(S(\lambda)^*S'(\lambda)\right) \, - p_n(\lambda) \right) \, \d \lambda -\beta_n(V)+\frac{1}{2\pi i} \int_0^\infty \textup{Tr}(\sigma(\lambda)^*\sigma'(\lambda)) \, \d \lambda,
\end{align}
where $\sigma$ is a low-energy correction in dimension $n = 1,3$ and otherwise $\sigma = \textup{Id}$. Here
\begin{align*}
p_n(\lambda) &= \sum_{j=1}^{\lfloor \frac{n-1}{2} \rfloor} c_j(n,V) \lambda^{\frac{n}{2}-j-1}
\end{align*}
is as in Definition \ref{defn:high-energy-poly1} and $\beta_n(V)$ is as defined in Equation \eqref{eq:truncated}.
\end{prop}
\begin{proof}
Let $\tilde{\chi}, \chi$ be as in Lemma \ref{lem:high-energy-nbig} and let $\beta$ be the corresponding high-energy correction for $S$ and $\sigma$ the corresponding low-energy correction. Then we have
\begin{align*}
&\textup{Index}(W_S) = \textup{Index}(W_{S\beta\sigma^*}) \\
&= \frac{1}{2\pi i} \int_0^\infty \left( \frac{\frac{\d}{\d \lambda} \textup{Det}(S(\lambda)\beta(\lambda)\sigma(\lambda))}{\textup{Det}(S(\lambda)\beta(\lambda)\sigma(\lambda))} \right) \, \d \lambda \\
&= \frac{1}{2\pi i} \int_0^\infty \left( \textup{Tr}\left(S(\lambda)^*S'(\lambda) \right) + i \frac{\d}{\d \lambda} \chi(\lambda^\frac12)\textup{Tr}(\tilde{A}(\lambda)) \right) \, \d \lambda + \frac{1}{2\pi i} \int_0^\infty \textup{Tr}(\sigma(\lambda)^*\sigma'(\lambda)) \, \d \lambda.
\end{align*}
By adding zero we account for the high-energy behaviour of $\textup{Tr}\left(S(\cdot)^*S'(\cdot)\right)$ using Lemma \ref{lem:ssf-L1} and Definition \ref{defn:high-energy-poly1} to obtain
\begin{align*}
&\frac{1}{2\pi i} \int_0^\infty \left( \textup{Tr}\left(S(\lambda)^*S'(\lambda) \right) + i \frac{\d}{\d \lambda} \chi(\lambda^\frac12)\textup{Tr}(\tilde{A}(\lambda)) \right) \, \d \lambda \\ &= \frac{1}{2\pi i} \int_0^\infty \left( \textup{Tr}\left(S(\lambda)^*S'(\lambda) \right) -p_n(\lambda) \right) + \left(p_n(\lambda)+ i \frac{\d}{\d \lambda} \chi(\lambda^\frac12)\textup{Tr}(\tilde{A}(\lambda)) \right) \, \d \lambda \\
 &= \frac{1}{2\pi i} \int_0^\infty \left( \textup{Tr}\left(S(\lambda)^*S'(\lambda) \right) -p_n(\lambda) \right) \, \d \lambda \\
&\quad + \frac{1}{2\pi i} \int_0^\infty \left(p_n(\lambda)+i \frac{\d}{\d \lambda} \left( \chi(\lambda^\frac12) \textup{Tr}(\tilde{A}(\lambda)) \right) \right)\, \d \lambda.
\end{align*}
Using Theorem \ref{thm:high-energy-traces} (see also Remark \ref{rmk:choice}) we write
\begin{align*}
i \textup{Tr}(\tilde{A}(\lambda)) &=   -P_n(\lambda) - E(\lambda)
\end{align*}
where $E$ is differentiable, $E(\lambda) = O(\lambda^{-\frac12})$ as $\lambda \to \infty$, 
\begin{align*}
P_n(\lambda) &= \sum_{j=1}^{\lfloor \frac{n}{2} \rfloor } C_j(n,V) \lambda^{\frac{n}{2}-j}, \quad \textup{ and } \quad P_n'(\lambda) = p_n(\lambda).
\end{align*}
Thus we find
\begin{align*}
&\frac{1}{2\pi i} \int_0^\infty \left( \textup{Tr}\left(S(\lambda)^*S'(\lambda) \right) + i \frac{\d}{\d \lambda} \chi(\lambda^\frac12)\textup{Tr}(\tilde{A}(\lambda)) \right) \, \d \lambda \\ &= \frac{1}{2\pi i} \int_0^\infty \left( \textup{Tr}\left(S(\lambda)^*S'(\lambda) \right) -p_n(\lambda) \right) \, \d \lambda \\
&\quad + \frac{1}{2\pi i} \int_0^\infty \left( p_n(\lambda)  -\frac{\d}{\d \lambda} \left( \chi(\lambda^\frac12) (P_n(\lambda)+E(\lambda))\right) \right)\, \d \lambda \\
 &= \frac{1}{2\pi i} \int_0^\infty \left( \textup{Tr}\left(S(\lambda)^*S'(\lambda) \right) -p_n(\lambda) \right) \, \d \lambda \\
&\quad +\frac{1}{2\pi i} \int_0^\infty \frac{\d}{\d \lambda} \left( P_n(\lambda) \left(1-\chi(\lambda^\frac12) \right) -\chi(\lambda^\frac12) E(\lambda) \right) \, \d \lambda \\
&= \frac{1}{2\pi i} \int_0^\infty \left( \textup{Tr}\left(S(\lambda)^*S'(\lambda) \right) -p_n(\lambda) \right)  - \frac{1}{2\pi i} P_n(0),
\end{align*}
where we have used the properties of $\chi$ as defined in Lemma \ref{lem:high-energy-nbig}.
Observing that $P_n(0) = 2\pi i\beta_n(V)$ completes the proof.
\end{proof}

We can now prove Levinson's theorem in all dimensions.

\begin{thm}
Suppose that $V \in C_c^\infty(\R^n)$ with corresponding scattering operator $S$. Then the number of bound states $N$ of $H = H_0+V$ is given by
\begin{align}
- N &= \textup{Index}(W_-) = \frac{1}{2\pi i} \int_0^\infty \left( \textup{Tr}\left(S(\lambda)^*S'(\lambda) \right) - p_n(\lambda) \right) \, \d \lambda - \beta_n(V)+N_{res},
\end{align}
where $p_n$ is as in Definition \ref{defn:high-energy-poly1}, $\beta_n(V)$ is as defined in Equation \eqref{eq:truncated} and 
\begin{align*}
N_{res} &= \begin{cases} \frac12, & \quad \textup{ if } n=1 \textup{ and there are no resonances,} \\
\textup{Index}(W_{res}), & \quad \textup{ if } n = 2,4, \\
\frac12, & \quad \textup{ if } n = 3 \textup{ and there exists a resonance}, \\
0, & \quad \textup{ otherwise}.
\end{cases}
\end{align*}
\end{thm}
\begin{proof}
For $n \geq 5$ this is the statement of Proposition \ref{prop:index-WS-nbig} combined with the observation that $\textup{Index}(W_-) = \textup{Index}(W_S) = -N$ of Proposition \ref{prop:index-wave-op} and Theorem \ref{thm:wave-op-form}. For $n = 1,3$ we use the additional observation (see Lemmas \ref{lem:sigma-1d} and \ref{lem:3d-low-energy}) that for a low energy correction $\sigma$ we have
\begin{align*}
\frac{1}{2\pi i} \int_0^\infty \textup{Tr} (\sigma(\lambda)^*\sigma'(\lambda)) \, \d \lambda &= \begin{cases}
\frac12, & \quad \textup{ if } n=1 \textup{ and there are no resonances,} \\
\frac12, & \quad \textup{ if } n = 3 \textup{ and there exists a resonance}, \\
0, & \quad \textup{ otherwise}.
\end{cases}
\end{align*}
 For $n = 2,4$ we observe that $\textup{Index}(W_-) = \textup{Index}(W_S)+\textup{Index}(W_{res})$ by Theorem \ref{thm:wave-op-form} and the result follows from Lemma \ref{lem:high-energy-nbig}.
\end{proof}

In particular, we have the following low-dimensional statements of Levinson's theorem. In dimension $n = 1$, we recover Levinson's original result \cite{levinson49} that
\begin{align*}
-N &= \frac{1}{2\pi i} \int_0^\infty \textup{Tr}(S(\lambda)^*S'(\lambda)) \, \d \lambda + \begin{cases} \frac12, \quad & \textup{ if there are no resonances}, \\
0, \quad & \textup{ otherwise}.
\end{cases}
\end{align*}
In dimension $n = 2$ we obtain the result of Boll\'{e} et al. \cite[Theorem 6.3]{bolle88},
\begin{align*}
-N &=\frac{1}{2\pi i} \int_0^\infty \textup{Tr}(S(\lambda)^*S'(\lambda)) \, \d \lambda +\frac{1}{4\pi} \int_{\R^2} V(x)\, \d x + N_{res}.
\end{align*}
In dimension $n = 3$ we recover the result of Boll\'{e} and Osborn \cite{bolle77},
\begin{align*}
-N &= \frac{1}{2\pi i} \int_0^\infty \left(\textup{Tr}(S(\lambda)^*S'(\lambda)) +\frac{i \lambda^{-\frac12}}{4\pi} \int_{\R^3} V(x)\, \d x \right) \, \d \lambda + \begin{cases} \frac12, \quad & \textup{ if there is a resonance}, \\
0, \quad & \textup{ otherwise}.
\end{cases}
\end{align*}
In dimension $n = 4$ we obtain the result of Jia, Nicoleau and Wang \cite[Theorem 5.3]{jia12},
\begin{align*}
-N &= \frac{1}{2\pi i} \int_0^\infty \! \left(\textup{Tr}(S(\lambda)^*S'(\lambda))\!+\! \frac{(2\pi i) \textup{Vol}(\Sf^3)}{2(2\pi)^4} \!\int_{\R^4} V(x) \d x \right)\! \d \lambda - \frac{\textup{Vol}(\Sf^3)}{4(2\pi)^4} \int_{\R^4} V(x)^2 \, \d x+N_{res}
\end{align*}
Finally, as a result of Levinson's theorem and the discussion in Remark \ref{rmk:heat-levinson} we deduce the low-energy behaviour of the spectral shift function.
\begin{cor}
Suppose that $V \in C_c^\infty(\R^n)$ with corresponding scattering operator $S$. Then for all $\lambda \in \R$ we have
\begin{align*}
\xi(\lambda) &= -\sum_{k=1}^K M(\lambda_k) \chi_{[\lambda_k,\infty)}(\lambda) - N_{res} \chi_{[0,\infty)}(\lambda) \\
&\quad - \chi_{[0,\infty)}(\lambda) \frac{1}{2\pi i} \int_0^\lambda \left( \textup{Tr}\left( S(\mu)^*S'(\mu) \right) - p_n(\mu)\right) \, \d \mu,
\end{align*}
where the distinct eigenvalues of $H$ are listed in numerical order $\lambda_1 < \cdots < \lambda_K \leq 0$ with each eigenvalue $\lambda_j$ having multiplicity $M(\lambda_j)$. In particular we have $\xi(0+) = -N - N_{res}$.
\end{cor}

%

\end{document}